\documentclass[11pt]{article}
\usepackage{amsfonts}
\usepackage{amssymb,latexsym}
\usepackage{amsmath,amscd}
\usepackage{theorem}
\usepackage{color}
\usepackage{combelow}
\usepackage[all]{xy}
\usepackage[pdftitle={Paper}, pdfauthor={Alejandro Cabrera}]{hyperref}
\usepackage{enumitem}
\usepackage{tensor}
\usepackage[affil-it]{authblk} 

\newtheorem{lemma} {Lemma} [section]
\newtheorem{proposition} [lemma] {Proposition}

\newtheorem{convention} [lemma] {Convention}
\newtheorem{theorem} [lemma] {Theorem}
\newtheorem{corollary} [lemma] {Corollary}
\newtheorem{definition}[lemma] {Definition}

\theorembodyfont{\rm}
\newtheorem{example}[lemma] {Example}
\newtheorem{remark}[lemma]{Remark}
\newenvironment{proof}{{\sc Proof:}}{{\hspace*{\fill} $\square$\\}}

\numberwithin{}{}

\newcommand{\X}{\ensuremath{\mathfrak{X}}}
\newcommand{\g}{\mathfrak{g}}

\newcommand{\Ad}{\text{\rm Ad}\,}       
\newcommand{\ad}{\text{\rm ad}\,}       

\newcommand{\dto}{\dashrightarrow}
\newcommand{\rmap}{\to}

\def\bea{\begin{eqnarray}}
\def\eea{\end{eqnarray}}
\def\bl{\begin{lemma}}
\def\el{\end{lemma}}
\def\br{\begin{remark}}
\def\er{\end{remark}}

\def\rr{\rightrightarrows}

\def\R{\mathbb{R}}

\def\s{\sigma}
\def\t{\tau}
\def\e{\epsilon}

\def\br{\bar{\rho}}

\def\hexp{\widehat{\exp}_V}

\def\VE{\mathfrak{ve}}
\def\hg{\hat{\gamma}}

\def\comas{,..,}

\begin{document}

\title{On local integration of Lie brackets}
\author{Alejandro Cabrera}
\affil{Departamento de Matem\'atica Aplicada, Instituto de Matem\'atica\\
Universidade Federal do Rio de Janeiro\\
 Caixa Postal 68530, Rio de Janeiro RJ 21941-909, Brasil\\ alejandro@matematica.ufrj.br}

\author{Ioan M\u{a}rcu\cb{t}}
\affil{Radboud University Nijmegen, IMAPP\\ 
6500 GL, Nijmegen, The Netherlands\\
i.marcut@math.ru.nl}
 
\author{Mar\'ia Amelia Salazar}
\affil{Instituto de Matem\'atica e Estat\'istica, GMA\\ 
Universidade Federal Fluminense\\ 
Rua Professor Marcos Waldemar de Freitas Reis s/n, Gragoat\'a, Niter\'oi, Rio de Janeiro, Brasil, 24.210-201\\ mariasalazar@id.uff.br}

\maketitle

\begin{abstract}
We give a direct, explicit and self-contained construction of a local Lie groupoid integrating a given Lie algebroid which only depends on the choice of a spray vector field lifting the underlying anchor map. This construction leads to a complete account of local Lie theory and, in particular, to a finite-dimensional proof of the fact that the category of germs of local Lie groupoids is equivalent to that of Lie algebroids.
\end{abstract}

\tableofcontents

\section{Introduction}\label{sec:intro}

It is not entirely true that \emph{``Lie algebroids are to Lie groupoids as Lie algebras are to Lie groups''} (as a general reference to Lie algebroid and Lie groupoid theory, see \cite{Mac}). The main difference is that the Lie functor, which differentiates Lie groupoids to Lie algebroids, is not (essentially) surjective, i.e.\ there are Lie algebroids which do not admit any integration by a smooth Lie groupoid (see \cite{AlmMol,CF1}). However, the non-integrability issue is of a global nature, and it can be overcome by considering \textbf{local Lie groupoids} instead (i.e.\ the structure maps are defined only locally, around the units - see Subsection \ref{subsection:defi_loc_lie}). In fact, the following holds:
\begin{theorem}\label{theo:equiv_cats}
Consider the category $\mathrm{loc Lie Grpd}$, whose objects are local Lie groupoids and whose morphisms are germs around the unit section of Lie groupoid morphisms, and the category $\mathrm{Lie Algd}$ of Lie algebroids. The Lie functor $\mathrm{Lie}:\mathrm{loc Lie Grpd}\rmap \mathrm{Lie Algd}$ induces an equivalence between these categories.
\end{theorem}

This result entails two parts:
\begin{itemize}
\item[(A)] Given two local Lie groupoids, every morphism between their respective Lie algebroids can be integrated to a local Lie groupoid map, and moreover, the germ around the unit section of this integration is unique.
\item[(B)] Every Lie algebroid is (isomorphic to) the Lie algebroid of a local Lie groupoid.
\end{itemize}

Validity of Theorem \ref{theo:equiv_cats} has been already  assumed since the work of Pradines in the 1960s (see \cite{Pradines0,Pradines,Pradines2}), however, part (B) was proven only later. In the setting of Poisson geometry, Coste \emph{et al.}\ \cite{CDW} constructed a local symplectic groupoid integrating a given Poisson manifold, by using the existence of symplectic realizations. Most likely, this result, when applied to the linear Poisson structure on the dual of a Lie algebroid, can be used to prove existence of local integrations in general; however, we are not aware of any written account of such an attempt. To our knowledge, the first complete proof of part (B) of Theorem \ref{theo:equiv_cats} appeared in the work of Crainic and Fernandes on integrability of Lie algebroids (see \cite[Corollary 5.1]{CF1}). The proof is a by-product of the main construction of their paper: the \emph{Weinstein groupoid} of a Lie algebroid. Namely, they prove that a submanifold of the Banach manifold of $C^1$-$A$-paths, which is transverse to the $A$-homotopy foliation and contains the unit section, inherits a local Lie groupoid structure with multiplication coming from concatenation of $A$-paths.
\medskip

The main result of this paper (Theorem \ref{theorem:thm2}) consists of an explicit construction of a local Lie groupoid integrating a given Lie algebroid, yielding a direct proof of part (B) of Theorem \ref{theo:equiv_cats}. For completeness, a proof of the more standard part (A) following similar methods is also included (Theorem \ref{thm:integr_morph}). 

Given a Lie algebroid $(A,[\cdot,\cdot],\rho)$, the input for our construction is a 
\textbf{Lie algebroid spray} consisting of a vector field $V\in \mathfrak{X}(A)$ which lifts the anchor: at $a\in A$, the tangent vector $V_a$ projects along $q:A\to M$ to $\rho(a)$. Such a spray can be constructed as the horizontal lift of $\rho(a)$ with respect to an ordinary linear connection on $A$ or, more generally, with respect to an \emph{$A$-connection} on $A$ (see \cite{CF1}). The first step in the construction of the local Lie groupoid is to realize its (right-invariant) Maurer-Cartan form. Concretely, using the spray $V$, one builds a Lie algebroid morphism
\[\theta:T^qU\rmap A,\ \ \ T^qU:=\ker(dq)\subset TU,\]
over an open neighborhood $U\subset A$ of the zero section, which is a fiber-wise isomorphism. The needed ``realization'' $\theta$ of $A$ was given via an explicit formula in \cite{Ori}; here, a different approach to the same construction is presented. The \textbf{local spray Lie groupoid} $G_V\rr M$ associated to $V$ has as total space an open neighborhood $G_V\subset A$ of the zero section that is sufficiently small so that the following structure maps are defined:
\begin{itemize}
\item the unit section $u:M\to G_V$ is the zero section;
\item the source map $\sigma:G_V\to M$ is the restriction of the bundle projection $q:A\to M$;
\item the target map is $\tau=q\circ \phi_V^1$, where $\phi_V^t$ denotes the local flow of $V$;
\item the inversion map is $\iota=-\phi_V^{1}$;
\item the multiplication $\mu:G_V\tensor[_\sigma]{\times}{_\tau}G_V\to U$ is defined as $\mu(a,b)=k_1$, where $k_t\in A_{q(b)}$, $t\in [0,1]$, is the solution to the ODE:
    \[\frac{dk_t}{dt}=\theta_{k_t}^{-1}(\phi_V^t(a)),\ \ k_0=b.\]
\end{itemize}

Theorem \ref{theorem:thm2} states that these maps define a local Lie groupoid with Maurer-Cartan form $\theta$, thus, providing an explicit proof for the existence of a local Lie groupoid integrating a given Lie algebroid.

In the case of a Lie algebra $A=\g$, with spray $V=0$, Theorem \ref{theorem:thm2} produces the local Lie group around $0\in\g$ defined by the Baker-Campbell-Hausdorff formula; see Example \ref{ex:Lie_groups}.

The spray construction can be used to provide explicit formulas for integrating various \emph{geometric structures} from infinitesimal data on the Lie algebroid to multiplicative local structures on the corresponding local Lie groupoids; this will be detailed in \cite{CMSbis}.

\medskip

{\bf Spray groupoids vs the space of $A$-paths.}
The constructions presented here are connected to the Crainic-Fernandes approach to integrability (see \cite{CF1}). Namely, any $A$-connection $\nabla$ on $A$ and any local Lie groupoid $G\rr M$ integrating $A$ determine an exponential map $\exp_{\nabla}:U\to G$ as in \cite{CF1}, which is a diffeomorphism between an open neighborhood $U\subset A$ of $M$ and its image. Transporting back the structure maps of $G$ along $\exp_{\nabla}$, one obtains a local Lie groupoid. Remarkably, the germ of this local Lie groupoid is independent of $G$, and is intrinsic to $A$ and $\nabla$. Moreover, it coincides with the local spray groupoid $G_V\rr M$ (see Corollary \ref{coro:exp}); this fact is non-trivial and is supported by Theorems \ref{theorem:thm2} and \ref{thm:integr_morph}. The above argument would yield a proof of Theorem \ref{theorem:thm2} if one assumes the existence of integrating local Lie groupoids; however, we do not follow this path, but give explicit formulas for the groupoid operations and prove directly the validity of all the axioms.

On the other hand, a possible candidate for an integrating local Lie groupoid is the one constructed in \cite{CF1} from the infinite-dimensional manifold $P(A)$ of $C^1$-$A$-paths. Namely, the path-exponential map corresponding to a Lie algebroid spray $V$,
\begin{equation}\label{eq:hatexp}
 \hexp: U \to P(A),\ \  a \mapsto (t\mapsto \phi_V^t(a))_{t\in [0,1]},
\end{equation}
is defined on a neighborhood $U\subset A$ of $M$. As shown in \cite{CF1}, after shrinking $U$, $\hexp$ is transverse to the $A$-homotopy foliation $\mathcal{F}$ and there is an induced local Lie groupoid structure on $G:=\hexp(U)$ integrating $A$ \cite[Corollary 5.1]{CF1} with multiplication coming from concatenation of $A$-paths, as in the construction of the Weinstein groupoid $G(A)=P(A)/\mathcal{F}$ associated to $A$.\\

{\bf Contents.} In Section \ref{sec:local}, general terminology related to local Lie groupoids and Lie algebroids is introduced, and Theorem \ref{thm:integr_morph} on the local integration of Lie algebroid morphisms is proven. Section \ref{sec:spray} presents the construction of the spray groupoid (Theorem
\ref{theorem:thm2}) and gives the explicit description of integrations of morphisms whose domain is a spray groupoid (Corollary \ref{corollary:morphisms}). 

\medskip

\textbf{Acknowledgments.} The authors would like to thank Marius Crainic, Pedro Frejlich, Rui Loja Fernandes, Marco Gualtieri, Eckhard Meinrenken and Daniele Sepe for useful discussions. The authors would like to thank the anonymous referee for his or her careful reading and suggestions. I.M.\ was supported by the NWO Veni grant 613.009.031 and the NSF grant DMS 14-05671. M.A.S.\ was a Post-Doctorate at IMPA during part of this project, funded by CAPES-Brazil. A.C. would like to thank CNPq and FAPERJ for financial support.

\section{Local integration of Lie algebroid morphisms}\label{sec:local}

In this section we describe an explicit integration procedure for Lie algebroid morphisms based on {\bf Maurer-Cartan forms} and {\bf{tubular structures}} for local Lie groupoids. This construction provides a proof of the fact that the Lie functor taking germs of local Lie groupoid maps to Lie algebroid morphisms is full and faithful. 

\subsection{Local Lie groupoids: notations and conventions}\label{subsection:defi_loc_lie}

Let us begin by explaining what we mean by a local Lie groupoid. Consider a manifold $G$, the space of arrows, a closed embedded submanifold $M\subset G$, the space of units, and a submersion $\sigma:G\to M$, the source map, such that $\sigma|_M=\mathrm{id}_M$. The structure of a \textbf{local Lie groupoid} on $(G,M,\sigma)$ consists of maps
\[\mathrm{target}\ \  \tau:G_{\tau}\to M, \ \ \ \mathrm{inversion}\ \ \iota: G_{\iota}\to G, \ \ \ \mathrm{multiplication}\ \ \mu:G_{\mu}\to G,\]
where $G_{\tau},G_{\iota}\subset G$ are open neighborhoods of $M$, $\tau$ satisfies $\tau|_M=\mathrm{id}_M$, $G_{\mu}\subset G\tensor[_\sigma]{\times}{_\tau}G_{\tau}$ is an open neighborhood of $M\cong \{(x,x) : x\in M\}$, such that the axioms of a Lie groupoid are satisfied locally around $M$. To make this precise, note that each groupoid axiom can be written as the equality of two maps defined on $G^{(k)}=\{(g_1,\ldots,g_k): \sigma(g_i)=\tau(g_{i+1})\}$; for example:
\[\tau\circ \iota=\sigma:G\to M, \ \ \ \ \tau\circ \mu=\tau\circ\mathrm{pr}_1:G^{(2)}\to M,\]
\[\mu\circ(\mu\times\mathrm{id}_G)=\mu\circ (\mathrm{id}_G\times\mu):G^{(3)}\to G.\]
We are assuming that each axiom holds an neighborhood of $M\cong\{(x,\ldots,x): x\in M\}$ in $G^{(k)}$.


The more standard definition of a local Lie groupoid \cite{CDW,VEst} assumes that the target map and the inversion map are globally defined, and also that some of the axioms hold globally. However, the germ of any local groupoid (in our sense) can be represented by a local groupoid in the classical sense (of \emph{loc.cit.}):

\begin{proposition}
For a local Lie groupoid structure on $(G,M,\sigma)$ there is an open neighborhood $U$ of $M$ in $G$ satisfying
\[U\subset G_{\tau}\cap G_{\iota},\ \ \iota(U)=U, \ \ U\tensor[_\sigma]{\times}{_\tau}U\subset G_{\mu},\]
and such that for $g,h\in U$ the following axioms hold:
\[\iota(\iota (g))=g, \ \ \tau(g)=\sigma(\iota(g)), \ \ \mu(\tau(g),g)=g=\mu(g,\sigma(g)),\]
\[\mu(g,\iota(g))=\tau(g), \ \ \mu(\iota(g),g)=\sigma(g), \ \ \sigma(\mu(g,h))=\sigma(h),\]
and, letting $U_{\mu}:=\{(g,h)\in U\tensor[_\sigma]{\times}{_\tau} U\ :\ \mu(g,h)\in U\}$, for $(g,h)\in U_{\mu}$ the following hold:
\[(\iota(h),\iota(g))\in U_{\mu}, \ \ \ \iota(\mu(g,h))=\mu(\iota(h),\iota(g)), \ \ \ \tau(\mu(g,h))=\tau(g),\]
and if the elements $(g,h)$, $(\mu(g,h),k)$ and $(h,k)$ belong to $U_{\mu}$, then also $(g,\mu(h,k))$ belongs to $U_{\mu}$ and we have that:
\[\mu(g,\mu(h,k))=\mu(\mu(g,h),k).\]

Moreover, the collection of such open sets $U$ forms a basis of neighborhoods of $M$ in $G$.
\end{proposition}
\begin{proof}
The proof is straightforward: start with any open neighborhood of $M$ in $G$ and shrink it step by step such that it satisfies all the properties. For example, to ensure that $\iota(U)=U$, and $\iota^2(g)=g$, take $U$ of the form $U=V\cap \iota(V)$, where $V$ is small enough such that $V,\iota(V)\subset G_{\iota}$ and $\iota^2|_V=\mathrm{id}_V$. An important point is to notice that open sets of the form $U^{(k)}=G^{(k)}\cap U^k$, where $U\subset G_{\tau}\cap G_{\iota}$ is an open neighborhood $M$, form a basis of neighborhoods of $M$ in $G^{(k)}$; this follows because $G^{(k)}$ has the subset topology of the product topology on $(G_{\tau}\cap G_{\iota})^k$.
\end{proof}

Any local Lie groupoid on $(G,M,\sigma)$ has an associated Lie algebroid $A\to M$, which is constructed exactly as in the case of ordinary Lie groupoids. As a vector bundle $A=\ker(d\sigma)|_M$, the anchor of $A$ is $\rho=d\tau|_A:A\to TM$, and the Lie bracket is obtained by identifying sections of $A$ with right invariant vector fields on $G$.

\begin{convention}\label{conv}
In order to keep notation as simple as possible, we will denote maps which are defined only in a neighborhood of $M$ with a dashed arrow $\dto$. In particular, the structure maps of a local groupoid on $(G,M,\sigma)$ are written as $\tau:G\dashrightarrow M$, $\iota:G\dashrightarrow G$ and $\mu:G\tensor[_\sigma]{\times}{_\tau}G\dto G$.
\end{convention}

\subsection{The Maurer-Cartan form}\label{sec:MC_form}

We recall the definition of the right-invariant Maurer-Cartan form $\theta_G$ of a local Lie groupoid $G$. Denote the vertical bundle of $\sigma$ by $T^{\sigma}G:=\ker(d \sigma)$. Given a smooth path $(-\epsilon,\epsilon)\to G$, $t\mapsto g_t$, such that $\sigma(g_t)=x$ for all $t$, its derivative at $t=0$ gives an element in $T^{\sigma}G$:
\[\frac{d}{dt}g_t\big|_{t=0}\in T^{\sigma}_{g_0}G,\]
and all elements in $T^{\sigma}G$ can be represented in this way. The \textbf{right-invariant Maurer-Cartan form} (or Maurer-Cartan form, for brevity) of $G$, denoted by $\theta_G$, is the following vector bundle map covering the target map:
\begin{equation}\label{eq:MC_diagram}\begin{array}{cc}
\xymatrixrowsep{0.4cm}
\xymatrixcolsep{1.2cm}
\xymatrix{
 T^{\sigma}G \ar@{-->}[r]^{\theta_G}\ar[d] & A\ar[d] \\
 G\ar@{-->}[r]^{\tau} & M}, &
\theta_G\Big(\frac{d }{dt}g_t\big|_{t=0}\Big):=\frac{d}{dt}g_tg_0^{-1}\big|_{t=0}\in A_{\tau(g_0)}.\end{array}\end{equation}
Notice that the domain of $\theta_G$ is of the form $T^{\sigma}U$, where $U\subset G$ is a neighborhood of $M$, and that it is right invariant: $\theta_G\Big(\frac{d }{dt}g_th\big|_{t=0}\Big)=\theta_G\Big(\frac{d }{dt}g_t\big|_{t=0}\Big)$, for any $h$ such that $\tau(h)=\sigma(g_t)=x$. 

The Maurer-Cartan form $\theta_G$ is a Lie algebroid map between the vertical distribution $T^{\sigma}G$ and $A$ (this condition can be written in the form of a Maurer-Cartan equation, see \cite{RuiIvan}). This Lie algebroid morphism integrates to a local groupoid map, called the {\bf division map}, given by:
\begin{equation}\label{division_map}
\xymatrixrowsep{0.4cm}
\xymatrixcolsep{1.2cm}
\xymatrix{
 G\tensor[_\sigma]{\times}{_\sigma}G \ar@{-->}[r]^{\delta}\ar@<-.5ex>[d] \ar@<.5ex>[d] & G\ar@<-.5ex>[d] \ar@<.5ex>[d] \\
 G \ar@{-->}[r]^{\tau} & M},\ \  \delta(g,h)=gh^{-1}.\end{equation}

\begin{remark}\label{rmk:factorizatgd}
The division map can be factorized as the following composition of groupoid maps
 \begin{equation}\label{eq:division_diagram}
\xymatrixrowsep{0.4cm}
\xymatrixcolsep{1.2cm}
\xymatrix{
 G\tensor[_\sigma]{\times}{_\sigma}G \ar@{^{(}->}[r]\ar@<-.5ex>[d] \ar@<.5ex>[d] & \sigma^!G\ar@{-->}[r]^{\Ad_{\beta}}\ar@<-.5ex>[d] \ar@<.5ex>[d] & \sigma^!G\ar@<-.5ex>[d] \ar@<.5ex>[d]\ar[r]^{\mathrm{pr}_2}& G\ar@<-.5ex>[d] \ar@<.5ex>[d] \\
 G\ar[r]^{\mathrm{id}}&G\ar@{-->}[r]^{\iota}&G \ar[r]^\sigma & M,}
 \end{equation}
 where $\sigma^!G=G\tensor[_\sigma]{\times}{_\tau}G\tensor[_\sigma]{\times}{_\sigma}G$ is the pullback groupoid of $G$ via the submersion $\sigma:G\to M$, with source $\bar\sigma(l,g,k)=k$, target $\bar\tau(l,g,k)=l$, and multiplication $(l,g,k)(k,h,e)=(l,gh,e)$ (see e.g.\ \cite{Mac}); where the inclusion $G\tensor[_\sigma]{\times}{_\sigma}G\hookrightarrow \sigma^!G$ maps $(g,h)$ to $(g,1_{\sigma(g)},h)$; and, using the canonical local bisection $\beta:G\dto\sigma^!G,\ \beta(g)=(g^{-1},g,g)$, the local groupoid map $\Ad_{\beta}$ is conjugation by $\beta$: $\Ad_{\beta}(l,g,k)=\beta(l)(l,g,k)\beta(k)^{-1}=(l^{-1},lgk^{-1},k^{-1})$. 
\end{remark}

\subsection{Tubular structures}\label{subsec:tubular}

In this subsection we discuss special tubular neighborhoods on local Lie groupoids.

Let $G$ be a local groupoid. Along the unit section $M\subset G$ the fibers of the source map give a natural decomposition of the tangent space $TG|_M=TM\oplus A$; in particular, the normal bundle of $M$ in $G$ is canonically identified with $A$.

A \textbf{tubular structure} on the local groupoid $G$ is a tubular neighborhood of $M$ in $G$ along the source fibers, i.e.\ an open embedding of bundles
\[
\xymatrixrowsep{0.4cm}
\xymatrixcolsep{1.2cm}
\xymatrix{
 A \ar@{-->}[r]^{\varphi}\ar[d]_{q} & G\ar[d]^{\sigma} \\
 M\ar[r]^{\mathrm{id}_M} & M}\]
where $q:A\to M$ denotes the projection, such that: $\sigma\circ\varphi=q$, $\varphi(0_x)=x$, for $x\in M$, and
\[\frac{d}{dt}\varphi(ta)\big|_{t=0}=(0,a)\in T_xM\oplus A_x=T_{x}G,\]
for all $a\in A$, where $x=q(a)$.

Note that $\varphi$ is defined on a neighborhood of $M$ in $A$. A tubular structure induces a scalar multiplication on $G$ along the source fibers
\[(t,g)\mapsto t g:=\varphi(t\varphi^{-1}(g)),\]
which is defined on a neighborhood of $\{0\}\times M$ in $\R\times G$. Moreover, the germ of $\varphi$ around $M$ can be recovered from this operation: $\varphi(\frac{d}{dt}tg|_{t=0})=g$. Clearly, any local Lie groupoid admits a tubular structure.

\subsection{Local integration of Lie algebroid morphisms}

A \textbf{local Lie groupoid map} between local Lie groupoids $G_1$ and $G_2$ is a smooth map \[F:G_1\dto G_2\]
defined on a neighborhood of $M_1\subset G_1$ that restricts to a map between the units
\[F_{M}:=F|_{M_1}:M_1\to M_2,\]
and that is multiplicative around $M_1$, i.e.\ $F(gh)=F(g)F(h)$ for all $(g,h)$ in an open neighborhood of $M_1$ in $G_1\tensor[_\sigma]{\times}{_\tau}G_1$. The local groupoid map $F$ induces a Lie algebroid morphism, denoted by
\[
\xymatrixrowsep{0.4cm}
\xymatrixcolsep{1.2cm}
\xymatrix{
 A_1 \ar[r]^{F_A}\ar[d] & A_2\ar[d] \\
 M_1\ar[r]^{F_M} & M_2}\]
where $A_1\to M_1$ and $A_2\to M_2$ are the respective Lie algebroids and $F_{A} = dF|_{A_1}$. In this case, we say that $F$ \textbf{integrates} the Lie algebroid morphism $F_{A}$.

Local Lie groupoids and germs of local Lie groupoid maps form a category. The Lie functor, taking $G$ to its Lie algebroid $A$, and a germ of a local Lie groupoid map $F$ to the induced Lie algebroid map $F_A$, is an equivalence to the category of Lie algebroids. This fact seems to be well-known; however, a complete reference is missing from the literature (see the bibliographical comments in the introduction). Our constructions below provide a detailed proof of this result. First, we show that the functor is full and faithful, i.e.\ Lie algebroid morphisms can be integrated to local Lie groupoid maps. Moreover, using a tubular structure on the domain, the integration can be made quite explicit. That the functor is essentially surjective will be shown in the next section, where we present a construction of a local integration of a Lie algebroid.

To address the integration of morphisms, let $G_i$ be local Lie groupoids with Lie algebroids $A_i$, $i=1,2$, and let $f:A_1 \to A_2$ be a Lie algebroid morphism covering $f_{M}:M_1\to M_2$. Fixing a tubular structure on $G_1$ and an element $g\in G_1$, the map $f$ together with the Maurer-Cartan forms $\theta_{G_i}$ give rise to the following ODE: 
\begin{equation}\label{eq:ODE_maps}
\theta_{G_2}\Big(\frac{d}{dt}k_t\Big)=f\circ \theta_{G_1}\Big(\frac{d}{dt}(tg)\Big),\ \ \ k_{0}=f_M(\sigma_1(g)),
\end{equation}
for a curve $t\mapsto k_t\in\sigma_2^{-1}(f_M(\sigma_1(g)))\subset G_2$.

\begin{theorem}\label{thm:integr_morph}
Let $G_1$ and $G_2$ be two local Lie groupoids with Lie algebroids $A_1$ and $A_2$, respectively, and $f:A_1\to A_2$ be a Lie algebroid morphism. There exists a local Lie groupoid map $F:G_1\dto G_2$ integrating $f$, and any two such integrations coincide around the unit section. 

Explicitly, fixing a tubular structure on $G_1$, we have that $F(g)=k_1$
where $k_t$ is the solution of the ODE \eqref{eq:ODE_maps} and $g\in G_1$ is taken close enough to the unit section such that the solution is defined for $t\in[0,1]$.
\end{theorem}

As an immediate consequence of Theorem \ref{thm:integr_morph}, we recover the well-known fact that the germ of an integration is essentially unique:

\begin{corollary}\label{corollary: all the same}
Any two local Lie groupoids integrating a given Lie algebroid are isomorphic around the unit section. 
\end{corollary}

An essential role in the proof of the above theorem is played by the following extract (Lemma \ref{Lemma_CF} below) of \cite[Proposition 1.3]{CF1} which describes variations of Lie algebroid paths. Before stating this result, we recall from \cite{CF1} that, given a Lie algebroid $A$, a time-dependent section $(t,x)\mapsto \alpha(t,x) \in A_x$ defines time-dependent flow of Lie algebroid automorphisms $\phi_\alpha^{t,s} \in \rm{Aut}(A)$ induced by the inner derivation $\ad_{\alpha}=[\alpha, -]$ of $A$:
\begin{equation}\label{eq:flowofsec}
 \frac{d}{dt}(\phi_\alpha^{t,s})^*\beta = (\phi_\alpha^{t,s})^*[\alpha^t,\beta], \ \forall \beta \in \Gamma(A), \ \phi_\alpha^{s,s} = \mathrm{id}, 
\end{equation}
where  the pull-back of sections is defined as $(\phi_\alpha^{t,s})^*\beta = \phi_\alpha^{s,t} \circ \beta \circ \phi_{\rho(\alpha)}^{t,s}$. The flow has also the multiplicative property:
\[\phi_\alpha^{t,s}\circ \phi_\alpha^{s,u} = \phi_\alpha^{t,u}.\]

\begin{remark}
When $A$ is the Lie algebroid of a local groupoid $G$, the flow of $\alpha$ can be naturally integrated to a flow by local Lie groupoid maps as follows (see \cite{CF1}). Let $\phi^{t,s}_{\alpha^R}$ denote the flow of the time-dependent right invariant vector field $\alpha^R\in\X(G)$ induced by $\alpha$. This flow gives rise to a family of bisections $\varphi^{t,s}_\alpha:=\phi^{t,s}_{\alpha^R}|_M$ of $G$ by restricting to $M\subset G$. Then, the local Lie groupoid map $\Ad_{\varphi^{t,s}_{\alpha}}:G\dto G$ defined by conjugation by $\varphi^{t,s}_{\alpha}$ integrates the flow of the time dependent section $\alpha$ of $A$:
\begin{equation}\label{eq:bisections}d(\Ad_{\varphi^{t,s}_{\alpha}})|_A=\phi^{t,s}_{\alpha}.\end{equation}
\end{remark}

We now state the results on variations of algebroid paths that we need in the proof of Theorem \ref{thm:integr_morph}.

\begin{lemma}\label{Lemma_CF}(\cite{CF1})
Let $q:A\to M$ be a Lie algebroid. Consider a smooth map
\[a:[0,1]\times J\to A, \ \ (t,\epsilon)\mapsto a_{\epsilon}(t),\]
where $J$ is an interval, and denote the base map by $\gamma_{\epsilon}(t):=q\circ a_{\epsilon}(t)$. Assume that $\gamma_{\epsilon}(0)=x$ for all $\epsilon\in J$, where $x\in M$ is fixed, and that $a_{\epsilon}$ is an $A$-path for all $\epsilon\in J$, i.e.\
$\label{eq:Apathdef} \rho\circ a_{\epsilon}(t)=\frac{d}{dt}\gamma_{\epsilon}(t),$
where $\rho:A\to TM$ denotes the anchor. Then there exists a unique smooth map
\[b:[0,1]\times J\to A, \ \ (t,\epsilon)\mapsto b_{\epsilon}(t)\]
such that $b_{\epsilon}(t)\in A_{\gamma_{\epsilon}(t)}$, $b_{\epsilon}(0)=0$, and such that the following is a Lie algebroid morphism:
\[f_{a,b}:T([0,1]\times J)\to A, \ \ f_{a,b}= a_{\epsilon}(t)dt+b_{\epsilon}(t)d\epsilon.\]

Moreover, $b$ can be explicitly constructed as follows:
\begin{enumerate}[label=(\arabic*)]
\item Let $\alpha_{\epsilon}(t)$ be a family of compactly supported sections of $A$ depending smoothly on $(t,\epsilon)\in [0,1]\times J$ such that $\alpha_{\e}(t,\gamma_{\epsilon}(t))= a_{\epsilon}(t)$. Denote the flow of the time dependent section $\alpha_\epsilon$ by
$\phi_{\alpha_\epsilon}^{t,s}(x) : A_x \to A_{y},  \ y:={\phi_{\rho(\alpha_\epsilon)}^{t,s}(x)}.$
Then $b$ is given by the integral formula:
\begin{equation}\label{eq:int_formula}
b_{\epsilon}(t)=\int_0^t\phi_{\alpha_\epsilon}^{t,s}(\gamma_{\epsilon}(s))\frac{d \alpha_{\epsilon}}{d\epsilon}(s,\gamma_{\epsilon}(s))d s.
\end{equation}
\item If $A$ is integrable by a Lie groupoid $G$, and $g:[0,1]\times J\to G $ is the solution to the ODE
    \[\theta_{G}\Big(\frac{d}{dt}g_{\epsilon}(t)\Big)=a_{\epsilon}(t), \ \ \ g_{\epsilon}(0)=1_x,\]
where $\theta_G$ denotes the Maurer-Cartan form of $G$, then $b$ is given by
    \[b_{\epsilon}(t)=\theta_{G}\Big(\frac{d}{d \epsilon}g_{\epsilon}(t)\Big).\]
\end{enumerate}
\end{lemma}

We now proceed to the proof of Theorem \ref{thm:integr_morph}.
\\

\begin{proof}[of Theorem \ref{thm:integr_morph}]
\noindent\emph{Step 1}. In this step, we show that the germ of $F$ around the unit section is unique.
Assume $F:G_1\dto G_2$ is a local groupoid map integrating $f$. Consider an open neighborhood $U$ of $M_1$ in $G_1$ such that for all $g\in U$ and all $s,t\in [0,1]$, we have that the following equality holds and, of course, all terms involved in the equality are defined:
\[F\left((sg)(tg)^{-1}\right)=F(sg)F(tg)^{-1}.\]
Using that $dF|_{A_1}=F_A=f$, applying $\frac{d}{ds}|_{s=t}$ to this equality, we obtain:
\[f\circ \theta_{G_1}\Big(\frac{d}{dt}(tg)\Big)=\theta_{G_2}\Big(\frac{d}{dt}F(tg)\Big).\]
This shows that the curve $t\mapsto F(tg)$ is the solution $t\mapsto k_t$ to the ODE \eqref{eq:ODE_maps}. The fact that $\theta_{G_2}:T_{h}^{\sigma}G_2\to A_{2,\tau(h)}$ is a linear isomorphism implies uniqueness of the solution to the ODE \eqref{eq:ODE_maps}, and therefore, uniqueness of $F$; namely $F(g)=k_1$, where $k_t$ is the (unique) solution of the ODE \eqref{eq:ODE_maps}.

\noindent\emph{Step 2}. We now go back to the definition of $F$ and show that, on an open neighborhood of $M_1$, the map $F$ satisfies $F^*(\theta_{G_2})=f\circ\theta_{G_1}$; in other words we have the following commutative diagram of Lie algebroid maps:
\begin{equation}\label{eq:claim}
\xymatrixrowsep{0.4cm}
\xymatrixcolsep{1.2cm}
\xymatrix{
 T^{\sigma}G_1 \ar@{-->}[r]^{\theta_{G_1}}\ar[d]_{dF} & A_1\ar[d]^{f} \\
 T^{\sigma}G_2\ar@{-->}[r]^{\theta_{G_2}} & A_2.}\end{equation}

First, we observe that for $g=x\in M_1$ the solution to the ODE is given by the constant path $k_t=x$, which is defined for all $t\in \R$. Therefore, there is a neighborhood $U$ of $M_1$ in $G$, such that for all $g\in U$ the ODE has a solution defined up to time $t=1$. On $U$ we thus have $F(g):=k_1$ well defined.

Let $g,v\in G_1$ be such that $\sigma_1(g)=\sigma_1(v)=x$. Taking $g$ in a sufficiently small neighborhood of $M_1$ and a sufficiently small open interval $J$ containing $0$, the tubular structure can be used to identify $g$ and elements of the form $\epsilon v$, where $\epsilon \in J$, as elements of $A_1|_x$. This, in turn, allows to make sense of the expression $t(g+\epsilon v)$ for any $t\in [0,1]$, obtaining then the following map:
\[v:[0,1]\times J\to G_1, \ \ (t,\epsilon)\mapsto v_{\epsilon}(t)=t(g+\epsilon v).\]
Note that, $d v_{\epsilon}(t)=\frac{d v_{\epsilon}(t)}{dt}dt +\frac{dv_{\epsilon}(t)}{d\epsilon} d\epsilon$ is a Lie algebroid morphism from $T([0,1]\times J)$ to $T^{\sigma}G_1$. Composing it with the Lie algebroid morphism $\theta_{G_1}:T^{\sigma}G_1\dto A_1$ and then with $f:A_1\to A_2$, we obtain that the following is also a Lie algebroid morphism:
\begin{equation}\label{map}f\circ \theta_{G_1}\Big(\frac{d}{dt} v_{\epsilon}(t)\Big)dt +f\circ \theta_{G_1}\Big(\frac{d}{d\epsilon} v_{\epsilon}(t)\Big)d\epsilon: T([0,1]\times J)\to A_2.\end{equation}

We now want to compute the components of \eqref{map} in terms of $F$. Denoting by $k_t(g)\in G_2$ the solution of \eqref{eq:ODE_maps}, we notice the following rescaling property: $k_{ts}(g) = k_s(t g)$ for $t,s \in [0,1]$ and $g\in U$. This follows by applying the chain rule on both sides of  \eqref{eq:ODE_maps} and by uniqueness of the solution of \eqref{eq:ODE_maps} for small $g$ (recall \emph{Step 1}).
Combining this property for $s=1$ with the definition of $F$ we get $F(v_\epsilon(t))=F(t(g+\epsilon v)) = k_t(g+\epsilon v)$ and the first component of \eqref{map} is then given by
\[f\circ \theta_{G_1}\Big(\frac{d}{dt} v_{\epsilon}(t)\Big)=\theta_{G_2}\Big(\frac{d}{dt} F(v_{\epsilon}(t))\Big).\]
Since $v_{\epsilon}(0)=x$, the second component in \eqref{map} vanishes for $t=0$, and therefore, by part (2) of Lemma \ref{Lemma_CF} with $G=G_2$ and $g_\epsilon(t)=F(v_\epsilon(t))$, it is given by:
\[f\circ \theta_{G_1}\Big(\frac{d}{d\epsilon} v_\epsilon(t)\Big)=
\theta_{G_2}\Big(\frac{d}{d\epsilon} F(v_{\epsilon}(t))\Big).\]
In particular, for $t=1$ and $\epsilon=0$, we obtain:
\[f\circ \theta_{G_1}\Big(\frac{d}{d\epsilon} (g+\epsilon v)|_{\epsilon=0}\Big)=
\theta_{G_2}\Big(\frac{d}{d\epsilon} F(g+\epsilon v)|_{\epsilon=0}\Big).\]
Since $v$ was arbitrary in the source fiber of $g$, this is precisely the commutativity of \eqref{eq:claim} at $g$.

\noindent\emph{Step 3}. We show that $F$ is multiplicative. Let $g,h\in G_1$ be two composable arrows that are close enough to the unit section. We show that the curves $t\mapsto F(tg)F(h)$ and $t\mapsto F((tg)h)$, which start at $F(h)$, both satisfy the ODE:
\[\theta_{G_2}\Big(\frac{d}{dt}k_t\Big)=f\circ \theta_{G_1}\Big(\frac{d}{dt}(tg)\Big);\]
hence the two curves must coincide for $t=1$, i.e.\ $F(gh)=F(g)F(h)$. For this we use right-invariance of both Maurer-Cartan forms and the relation $f\circ \theta_{G_1}=F^*(\theta_{G_2})$:
\[\theta_{G_2}\Big(\frac{d}{dt}F((tg)h)\Big)=f\circ \theta_{G_1}\Big(\frac{d}{dt}(tg)h\Big)=f\circ \theta_{G_1}\Big(\frac{d}{dt}(tg)\Big);\]
\[\theta_{G_2}\Big(\frac{d}{dt}F(tg)F(h)\Big)=\theta_{G_2}\Big(\frac{d}{dt}F(tg)\Big)=
f\circ \theta_{G_1}\Big(\frac{d}{dt}(tg)\Big).\]
\end{proof}

\begin{remark}
 The role of the tubular structure on $G_1$ above is to provide a path $g(t)$ joining $\sigma(g)$ to $g$ within the underlying source fiber, namely, $g(t) = tg$. For any such path, there is an associated ODE generalizing \eqref{eq:ODE_maps} which (uniquely) characterizes the value $F(g)$ of an integration of $f$, for $g$ close enough to the identities.
\end{remark}

\subsection{Local integration of cochains} \label{sec:cochains}

In this section we show how a tubular structure can be used to integrate Lie algebroid cochains to local Lie groupoid cochains. Let $G\rightrightarrows M$ be a local Lie groupoid endowed with a tubular structure with scalar multiplication $g \mapsto tg$, and let $A$ denote its Lie algebroid. 
Recall \cite{Cr,L-BM,WX} that Lie algebroid $p$-cochains are sections of $\wedge^p A^*$, while local Lie groupoid cochains are smooth functions defined on (small) composable arrows $f:G^{(p)} \dto \R$. The \emph{van Est map}, applied to such a local Lie groupoid cochain $f$, is defined as follows \cite{WX}:  
\[ \VE(f)(a_1 \comas a_p)(x) = \sum_{\pi \ perm.} \mathrm{sign}(\pi) D_{a_{\pi(p)}}\cdots D_{a_{\pi(1)}}f|_x, \ a_i \in \Gamma (A),\]
where $D_a f : G^{(p-1)}\dto \R$ is defined by $D_af(g_2 \comas g_p) = \frac{d}{d\e}|_{\e=0}f(h(\e),g_2 \comas g_p)$ for any curve $h(\e)\in \s^{-1}(\t(g_2))$ starting at $h(0)=\t(g_2)$ with velocity $h'(0)=a(\t(g_2))$. If $f$ is \emph{normalized}, i.e. $f(g_1\comas g_p)=0$ whenever one of the $g_i$'s is a unit, then $\VE(f)(a_1 \comas a_p)$ is $C^\infty(M)$-linear in the sections $a_i\in \Gamma (A)$, so that $\VE(f)$ indeed defines an element of $\Gamma (\wedge^p A^*)$.

We now use the tubular structure to produce a right inverse for $\VE$.
Given a $p$-tuple $(g_1\comas g_p)\in G^{(p)}$ of small composable arrows in $G$, we recursively define maps $\gamma_{g_1\comas g_p}: I^p \to \s^{-1}(\s(g_p))\subset G$, with $I=[0,1]$, by
\[ \gamma_{g_1}(t_1) = t_1 g_1, \ \gamma_{g_1 \comas g_{p}}(t_1 \comas t_{p})= t_{p} \mu(\gamma_{g_1 \comas g_{p-1}}(t_1 \comas t_{p-1}), g_{p}).\]
Let the map $\Psi: \Gamma (\wedge^p A^*) \to \{ f: G^{(p)}\dto \R\}$ be defined by
\[ \Psi(\alpha)(g_1 \comas g_p) =  \int_{I^p} \alpha\left( \theta(\partial_{1}\gamma_{g_1 \comas g_p}) \comas \theta({\partial}_{p}\gamma_{g_1 \comas g_p}) \right)  dt_1\dots dt_p = \int_{I^p} (\theta\circ d\gamma_{g_1 \comas g_p})^*\alpha,\]
where $\theta:=\theta_G$ is the Maurer-Cartan form of $G$ and $\partial_j \gamma$ is the partial derivative of $(t_1 \comas t_p) \mapsto \gamma(t_1 \comas t_p)$ with respect to its $j$-th variable.
Notice that, when one of the $g_i$'s is a unit, then $\Psi(\alpha)(g_1,..,g_p)=0$ since $\gamma_{g_1 \comas g_p}$ becomes a \emph{degenerate} $p$-cube as it factors through a map $I^p \to I^{p-1}$ and $\Omega^p(I^{p-1})=0$. Hence, $\Psi(\alpha)$ is normalized for any $\alpha\in \Gamma (\wedge^p A^*)$.

\begin{example} \label{ex:1cochains}\emph{($1$-cochains)}
Let $\alpha \in \Gamma (A^*)$ be an algebroid 1-cochain. Since $\gamma_g(t)=tg$, we have
\begin{equation}\label{eq:our_formula}\Psi(\alpha)(g) = \int_0^1 \alpha \circ \theta\left(\frac{d}{dt}tg\right) \ dt.\end{equation}
To differentiate this local groupoid $1$-cochain, we consider the curve $h(\e) = \e a$ so that
\[ \VE(\Psi(\alpha))(a) = \frac{d}{d\e}|_{\e=0}  \int_0^1 \alpha \circ \theta\left(\frac{d}{dt}t\e a\right) \ dt =  \alpha \circ \theta\left(\frac{d}{du}ua\right)|_{u=0}  = \alpha(a).\]
When $\alpha$ is \emph{exact}, namely $\alpha = \rho^*df$ ($=d_A f$ below) for $f\in C^\infty(M)$, then
\[ \Psi(d_Af)(g) =  \int_0^1 df \circ \rho \circ \theta \left(\frac{d}{dt}tg\right) \ dt = \int_0^1 \frac{d}{dt}f(\t(tg)) \ dt = f(\t(g)) - f(\s(g)),\]
where, in the last step, we used that the tubular structure $\varphi$ satisfies $\s\circ \varphi=q:A \dto M$.

 In \cite[Theorem 1.3]{WX} the authors prove that the van Est map is an isomorphism in degree 1 cohomology if $G$ is source simply connected, using an explicit formula $\Lambda (\alpha)$ for the inverse: $\Lambda (\alpha)(g)=\int _{0}^1\alpha^R(\frac{d}{dt}g(t))\ dt$, where the integration is over any path $g(t)$ in the $\sigma$-fiber joining $1_{\sigma(g)}$ and $g$, and $\alpha^R\in \Gamma((T^\sigma G)^*)$ is the right invariant foliated 1-form induced by $\alpha$\footnote{Due to the conventions for multiplication in \cite{WX}, left translation in \cite{WX} becomes right translation with our conventions.}. Of course the path $g(t):t\mapsto tg$ does the job and we recover the formula \eqref{eq:our_formula}.
\end{example}

The two properties described in the previous example generalize to arbitrary $p$-cochains:

\begin{proposition}\label{prop:ve}
The map $\Psi$ defined above satisfies
\begin{enumerate}
\item \label{e1}  $\Psi(d_A\alpha) = - \delta \Psi(\alpha)$, where  $d_A$ denotes the Chevalley-Eilenberg differential associated to $A$ and $\delta$ the differentiable cohomology differential for $G$ (see \cite{Cr} and the proof below);
\item \label{e2}  $\VE(\Psi(\alpha)) = \alpha$ for all $\alpha \in \Gamma (\wedge^p A^*)$.
\end{enumerate}
\end{proposition}
\begin{proof}
To show \ref{e1}, one first observes that $\theta\circ d\gamma_{g_1\comas g_p}: T I^p \to A$ is an algebroid morphism, being a composition of such; hence, it intertwines $d_A$ and the de Rham differential on $\Omega(I^p)$. Applying Stokes' theorem in the computation of $\Psi(d_A \alpha)$ yields
\[ \Psi(d_A\alpha)(g_1\comas g_{p+1}) = \int_{\partial I^{p+1}} (\theta\circ d\gamma_{g_1 \comas g_{p+1}})^*\alpha.\]
 The above integral splits as a sum over the faces of the $(p+1)$-cube $I^{p+1}$ which are determined by $t_i =0$ or $t_i =1$, for $i=1,..,p+1$, each one with its induced orientation. Let us then consider the restriction of $\gamma_{g_1\comas g_{p+1}}$ to such faces following its definition. The face $t_1=0$ yields the (negatively oriented) $p$-cube $\gamma_{g_2\comas g_{p+1}}$ since zero is mapped to the groupoid units through the tubular structure. The face $t_1=1$ yields the (positively oriented) $p$-cube $\gamma_{\mu(g_1,g_2),g_3 \comas g_{p+1}}$. For $1<i\leq p+1$, setting $t_i=0$ yields a degenerate $p$-cube since the underlying map does not depend on $t_j$ for $j<i$, the corresponding integrand thus vanishes and hence these terms do not contribute. The faces $t_i=1$ for $1<i<p+1$ yield, using associativity of $\mu$, the $p$-cubes $\gamma_{g_1\comas \mu(g_i,g_{i+1})\comas g_{p+1}}$ in which $g_i$ is replaced by $\mu(g_i,g_{i+1})$. Finally, the face $t_{p+1}=1$ yields $(t_1\comas t_p) \mapsto \mu(\gamma_{g_1\comas g_p}(t_1\comas t_p), g_{p+1})$ which, since $\theta$ is right invariant (and taking into account the induced orientation), contributes to the boundary integral as $(-1)^{p+2} \Psi(\alpha)(g_1\comas g_p)$. 

In conclusion, the sum over boundary terms yields
directly $-\delta \Psi(\alpha)$:
\[ \int_{\partial I^{p+1}} (\theta\circ d\gamma_{g_1 \comas g_{p+1}})^*\alpha  = - \Psi(\alpha)(g_2\comas g_{p+1}) + \sum_{i=1}^{p} (-1)^{i+1} \Psi(\alpha)(g_1\comas \mu(g_i,g_{i+1})\comas  g_{p+1}) + \]\[+ (-1)^{p+2} \Psi(\alpha)(g_1\comas g_p) = -(\delta \Psi(\alpha))(g_1\comas g_{p+1}),\]
and hence \ref{e1} is proven.

For \ref{e2}, we need to recursively differentiate $f=\Psi(\alpha)$ along sections $a_1,..,a_p \in \Gamma (A)$. Since both sides of the desired equality are multilinear forms on $A$, it is enough to show that they agree when evaluated on small enough $a_1\comas a_p \in \Gamma (A)$. For such sections, the underlying differentiation is carried out in Lemma \ref{lma:DvE} of Appendix \ref{appendix}. Taking $k=p$ in Lemma \ref{lma:DvE} and noting that by the properties of the tubular structure we have $v^a: = \theta\left(\frac{d}{dt}ta|_{t=0}\right)=a $ (cf.\ Definition \ref{eq:defv}), we obtain
\[ D_{a_p}\cdots D_{a_1} \Psi(\alpha) = \frac{1}{p!} \alpha(a_1\comas a_p) \in C^\infty(M).\]
Then \ref{e2} follows directly from the above as the $1/p!$ cancels the sum over permutations in $\VE(\Psi(\alpha))$ since $\alpha$ is already alternating.
\end{proof}

 The map $\Psi$ above thus provides an explicit geometric description for integration of cochains, inverting differentiation through the van Est map and descending to cohomology.
The fact that the van Est map yields an isomorphism in cohomology for any local Lie groupoid was shown in \cite{L-BM} where, moreover, a retraction along the source  fibers is shown to induce an inverse for $\VE$ at the level of cochains by means of homological perturbation theory.

\section{The local Lie groupoid associated to a spray}\label{sec:spray}

In this section we describe an explicit, finite-dimensional construction of a local Lie groupoid integrating a given Lie algebroid, which uses a Lie algebroid spray. This construction provides a proof of the fact that the Lie functor, taking local Lie groupoids to Lie algebroids, is essentially surjective, thus, by the previous section, an equivalence of categories.


\subsection{Lie algebroid sprays}
We first recall the following notion:

\begin{definition}
Let $(A,[\cdot,\cdot],\rho)$ be a Lie algebroid. Let $q:A\to M$ denote the bundle projection, and let $m_t:A\to A$ denote scalar multiplication by $t\in \R$. A \textbf{Lie algebroid spray} for $A$ is a vector field $V$ on the manifold $A$ which satisfies:
\begin{enumerate}
\item $V$ is homogeneous of degree one: $m_t^*(V)=tV$ for all $t\neq 0$;
\item $V$ lifts the anchor of $A$ in the sense that: $dq(V_a)=\rho(a)\in TM$, for all $a\in A$.
\end{enumerate}
\end{definition}

The local flow of the spray will be denoted by $\phi_V^t\equiv \phi_V^{t,0}$. The second condition defining $V$ implies that $a(t):= \phi_V^t(a)\in A$ satisfies the condition $\frac{d}{dt}q( a(t)) = \rho(a(t))$ of an \emph{$A$-path} (in the sense of \cite{CF1}) for any initial condition $a\in A$. 

\begin{remark}
\begin{enumerate}[leftmargin=*]
\item The first condition implies that $V$ vanishes along the zero section of $A$. Hence, its flow is defined up to time $t=1$ on an open neighborhood of the zero section. We thus obtain the map $\hexp: A \dto P(A),\   a \mapsto (t\mapsto \phi_V^t(a))$ discussed in the last part of the introduction. 
\item A Lie algebroid spray is equivalent to a torsion-free $A$-connection on $A$ (for these notions, see \cite{Fe1}). Let $x^i$ denote local coordinates on $M$ and $u^a$ fiberwise linear coordinates on $A$. Then any spray has the local expression
\[V = \rho^i_a(x) u^a \partial_{x^i} + \Gamma^{c}_{ab}(x) u^au^b \partial_{u^c},\]
where $\rho^i_a$ denote the local coefficients of the anchor map $\rho:A \to TM$ and $\Gamma^{c}_{ab}$ define the Christoffel symbols for the $A$-connection. The torsion vanishes since only the symmetric part of these symbols contributes non-trivially to $V$.
\item Sprays always exist: for example, define $V_{a}\in T_a A$ to be the horizontal lift of $\rho(a)\in T_{q(a)}M$ with respect to a fixed linear ($TM$-)connection $\nabla$ on $A$.
\item That $V$ is homogeneous of degree one is equivalent to the following property of its local flow:
\begin{equation}\label{eq:flowV}\phi_V^t(sa)=s\phi_V^{st}(a).\end{equation}
\item The flow of a spray fixes points of the zero section $M\subset A$, namely $\phi_V^t(0_x)=0_x$. 
 Using the natural decomposition $TA|_M= TM \oplus A$ into vectors tangent to the zero section plus vertical vectors, the differential of $\phi_V^t$ at such points yields
\begin{equation}\label{eq:flowVatzero}
 d\phi_V^t: TA|_{M} = TM \oplus A \to TM \oplus A, \ \ ( u, a) \mapsto ( u + t \rho(a),a),
\end{equation}
where $d\phi_V^t(a)$ is computed using equation \eqref{eq:flowV} and the condition $\frac{d}{dt}q( \phi_V^t(a)) = \rho(\phi_V^t(a))$:
$$\frac{d}{ds}\phi_V^t(sa)|_{s=0}=\frac{d}{ds}s\phi_V^{ts}(a)|_{s=0}=\phi_V^{0}(a)+\frac{d}{ds}q(\phi_V^{ts}(a))|_{s=0}=a+\frac{d}{ds}q(\phi_V^{s}(ta))|_{s=0}=a+\rho(ta)$$
\item Consider the pullback Lie algebroid of $A$ via the submersion $q:A\to M$ (see e.g.\ \cite{HiggMack})
\[q^{!}A:=TA\times_{TM}A=\{(u,a)\in TA\times A \ : \ dq(u)=\rho(a)\}.\]
The Lie algebroid $q^!A$ is an extension of $A$ by the vertical foliation $T^qA:=\ker(dq)$: 
\begin{equation}\label{eq:diagram_extension}
\xymatrixrowsep{0.4cm}
\xymatrixcolsep{1.2cm}
\xymatrix{
 T^{q}A \ar@{^{(}->}[r]\ar[d]& q^!A \ar[r]^{\mathrm{pr}_2}\ar[d]&A\ar[d]^q \\
 A\ar[r]^{\mathrm{id}}&A\ar[r]^q & M,}
 \end{equation}
where the inclusion $T^qA\hookrightarrow q^!A$ sends $v\in T^q_aA$ to $(v,0_{q(a)})\in (q^!A)_a$. 
A spray $V$ on $A$ induces a section of $q^!A$,
 \[\widehat{V}\in \Gamma(q^!A), \ \ \ \widehat{V}_a:=(V_a,a), \ \ a\in A,\]
 which satisfies $\mathrm{pr}_2\circ \widehat{V}=\mathrm{id}_A$. The Euler vector field can be regarded as a section of $q^!A$:
 \[E\in \Gamma(T^qA)\subset \Gamma(q^!A),\ \ \ E_a:=\frac{d}{d t}(e^ta)\big|_{t=0}, \ \ a\in A.\]
 The homogeneity of $V$ is expressed as a simple equation in the Lie algebra $\Gamma(q^!A)$:
 \begin{equation}\label{eq:homo}
 [E,\widehat{V}]=\widehat{V}.
 \end{equation}
\end{enumerate}
\end{remark}

\begin{example}\label{ex:alg}
 Let $A=\g$ be a Lie algebra (i.e.\ a Lie algebroid over a point $M=pt$). Sprays are simply $1$-homogeneous vector fields on $\g$, since the second condition in the definition of $V$ is empty. Denoting by $u^a$ linear coordinates on $\g$, a spray $V$ must be of the form $ V = \Gamma^{c}_{ab} u^au^b \partial_{u^c}$ with $\Gamma^{c}_{ab}\in \mathbb{R}$. In particular, $V=0$ defines a spray for $\g$ whose flow is the identity $\phi_V^t=\mathrm{id}_{\g}$.
\end{example}

\begin{example}
An ordinary geodesic spray $V\in \X(TM)$ associated to a Riemannian metric on $M$ defines a Lie algebroid spray for the tangent Lie algebroid $A=TM$.
\end{example}

\subsection{Realization forms}

Let $(A,[\cdot,\cdot],\rho)$ be a Lie algebroid with bundle projection $q:A\to M$ and let $V$ be a Lie algebroid spray on $A$. 
A crucial step in the definition of the multiplication of the local groupoid is the construction of the corresponding right-invariant Maurer-Cartan form associated to $V$ obtained in \cite{Ori}, which is called here the {\it realization form}, and which plays the role of the Maurer-Cartan form \eqref{eq:MC_diagram} of this local groupoid  (see Theorem \ref{theorem:thm2}). We give here an independent account of the construction in \cite{Ori}, as well as a different proof, relying on the Lie algebroid $q^!A\to A$ and thinking of an infinitesimal analogue of a factorization like \eqref{eq:division_diagram} of the division map.

Let $\phi_{\widehat{V}}^t\equiv \phi_{\widehat{V}}^{t,0}$ denote the flow of the Lie algebroid (time-independent) section $\widehat{V}\in \Gamma(q^!A)$ (recall equation \eqref{eq:flowofsec}). Since $\widehat{V}$ is sent by the anchor $\mathrm{pr}_1:q^!A\to TA$ to $V$, the flow $\phi_{\widehat{V}}^t$ is a Lie algebroid automorphism covering $\phi_{V}^t$.

\begin{definition}
The \textbf{realization form associated to $V$}, $\theta:T^{q}A \dto A$, is the composition of the following Lie algebroid maps:
\begin{equation}\label{defi:MC_diagram}
\xymatrixrowsep{0.4cm}
\xymatrixcolsep{1.2cm}
\xymatrix{
 T^{q}A \ar@{^{(}->}[r]\ar[d]& q^!A \ar@{-->}[r]^{\phi_{\widehat{V}}^1}\ar[d]&
 q^!A \ar[r]^{\mathrm{pr}_2}\ar[d]&
 A\ar[d]^q \\
 A\ar[r]^{\mathrm{id}}&A\ar@{-->}[r]^{\phi_{V}^1}&A \ar[r]^q & M}
 \end{equation}
\end{definition}
Thus, if $U$ an open neighborhood of $M$ in $A$ on which $\phi_V^t$ is defined for $t\in [0,1]$, then the realization form is the Lie algebroid morphism covering $\tau=q\circ\phi_{V}^1:U\to M$:
\begin{equation}\label{defi:MC_of_V }
\begin{array}{cc}
\xymatrixrowsep{0.4cm}
\xymatrixcolsep{1.2cm}
\xymatrix{
 T^{q}A|_U \ar[r]^{\theta}\ar[d] & A\ar[d] \\
 U\ar[r]^{\tau} & M}, &
\theta_a(v):=\mathrm{pr}_2\circ \phi_{\widehat{V}}^1(v).\end{array}
\end{equation}
for $a\in U$ and $v\in T^q_{a}A \cong A_{q(a)}$.

Next, we derive some immediate properties of $\theta$. Using equation \eqref{eq:homo} and that the operation of pulling sections back along $\phi_{\widehat{V}}^t$ behaves like exponentiating the derivation $[\widehat{V},\cdot]$ (recall equation \eqref{eq:flowofsec}), we obtain:
\[\frac{d}{dt}\left((\phi_{\widehat{V}}^t)^*(E+t\widehat{V})\right)=(\phi_{\widehat{V}}^t)^*([\widehat{V},E]+\widehat{V})=0\]
which implies that
\[(\phi_{\widehat{V}}^t)_*(E)=E+t\widehat{V}.\]
For $t=1$, this yields:
$\phi_{\widehat{V}}^1(E_a,0_{q(a)})=(E_{\phi^1_V(a)}+V_{\phi^1_V(a)},\phi^1_V(a))\in (q^!A)_{\phi^1_V(a)},$
for all $a\in U$. Projecting this onto the second component gives:
\begin{equation}\label{eq:theta_Euler}
\theta_a(a)=\phi_V^1(a), \ \ \textrm{for all}\ a\in U,
\end{equation}
where we used the identification $T^q_aA=A_{q(a)}$, under which $E_a=d/dt|_{t=0}(a + t a) \equiv a$. Equation \eqref{eq:theta_Euler} is a key relation between $\theta$ and $V$ to be repeatedly used in this section. 

A direct corollary comes from replacing $a$ by $t a$ in equation \eqref{eq:theta_Euler}, and using (\ref{eq:flowV}) to obtain:
\begin{equation}\label{eq:theta_Euler2}
\theta_{t a}\Big(\frac{d}{dt}(ta)\Big)=\theta_{t a}(a)=\phi_V^{t}(a).
\end{equation}
In the limit $t \rightarrow0$, we obtain that $\theta_{q(a)}(a)=a$, i.e.\ $\theta$ is the identity along $M$. Thus, by shrinking $U$ we may assume that $\theta$ is a fibre-wise linear isomorphism.

Compared to \cite{Ori}, our approach of introducing $\theta$ using $q^!A$ has the advantage that it makes it obvious that $\theta$ is a Lie algebroid map, being a composition of such. To verify that we indeed obtain the same object, one can recall from \cite{Ori} that the equation (\ref{eq:theta_Euler}) characterizes the Lie algebroid map $\theta$ uniquely around $M$.

The following lemma further characterizes $\theta$ as part of a Lie algebroid morphism, and recovers the formula from \cite{Ori}.

\begin{lemma}\label{lemma:old_theta}
Let  $U$ denote an open neighborhood of $M$ in $A$ on which $\phi_V^t$ is defined for $t\in [0,1]$, $a\in U$ and $v\in A_{q(a)}$. Consider a small enough interval $J$ around $0$ such that $t(a+\epsilon v)\in U$, for all $(t,\epsilon)\in [0,1]\times J$. Then, the map
\[\phi_V^t(a+\epsilon v)\ dt+ \theta_{t(a+\epsilon v)}(tv)\ d\epsilon: T([0,1]\times J)\rmap A\]
defines a Lie algebroid morphism over $[0,1]\times J \to M, (t,\epsilon) \mapsto \gamma_\epsilon(t):= q(\phi_V^t(a + \e v))$.
In particular, from  \eqref{eq:int_formula} we get that
\begin{equation}\label{eq:int_formula_theta}
\theta_{a}(v)=\int_0^1\phi_{\alpha_0}^{1,s}(\gamma_{0}(s))\frac{d \alpha_{\epsilon}}{d\epsilon}|_{\e=0}(s,\gamma_{0}(s))d s,
\end{equation}
where $\alpha_\e$ is a family of compactly supported, time-dependent sections of $A$ such that $\alpha_\e(t,\gamma_\e(t)) = \phi_V^t(a + \e v)$, with 
$\phi_{\alpha_\e}^{1,s}:A_{\gamma_\e(s)} \to A_{\gamma_\e(1)}$ denoting the flow of $\alpha_\e$.
\end{lemma}
\begin{proof}
Differentiating the map
$\hg:[0,1]\times J\rmap A_{q(a)},\  \hg(t,\epsilon):=t(a+\epsilon v)$,
we obtain a Lie algebroid morphism covering $\hg$
\[d\hg=\partial_{t}\hg\ dt +\partial_{\epsilon}\hg\ d\epsilon:T([0,1]\times J)\to T^qA.\]
Composing this with $\theta$, we obtain a Lie algebroid morphism into $A$:
\[\theta\circ d\hg=\theta_{\hg}(\partial_{t}\hg)\ dt +\theta_{\hg}(\partial_{\epsilon}\hg)\ d\epsilon:T([0,1]\times J)\to A\]
covering $\tau\circ \hg$. The first statement follows because
$\theta_{\hg}(\partial_t\hg)=\phi^t_V(a+\epsilon v)$ by \eqref{eq:theta_Euler2}
and $\partial_{\epsilon}\hg=tv$. Formula \eqref{eq:int_formula_theta} for $\theta_a(v)$ then follows directly by evaluating \eqref{eq:int_formula} at $t=1$ and $\e=0$.
\end{proof}

\begin{example}\label{ex:lie2}
 Let $\g$ be a Lie algebra endowed with the zero spray $V=0$. We evaluate the associated realization form $\theta: T\g \dto \g$ using formula \eqref{eq:int_formula_theta}. Given $a\in \g$ close to zero and $v\in \g$, since $\g$ is a Lie algebroid over a single point, then $\alpha_\e(t) = \phi_V^t(a+\e v) = a + \e v$ defines a section that can be used to evaluate \eqref{eq:int_formula_theta}. In particular, $\alpha_0(t) = a$ and $\frac{d}{d\e}|_{\e=0} \alpha_\e(t) = v$ (both time-independent). Moreover, using the definition \eqref{eq:flowofsec} of the associated flow, it follows that
 \[\phi_{\alpha_0}^{t,s} = e^{-(t-s) \ad_a}: \g \to \g.\]
 We can then directly compute $\theta$ using \eqref{eq:int_formula_theta}, yielding the well known expression 
 \begin{equation}\label{eq:thetaonLie}
 \theta_a(v) = \int_0^1 e^{-(1-s) \ad_a}(v) \  d s = \frac{e^{-\ad_a}-I}{-\ad_a}(v). 
 \end{equation}
For this see \cite[Section 1.5]{DK}, where $\ad_a$ appears instead of $-\ad_a$ due to a different convention for the Lie bracket; namely, our convention for the Lie bracket comes from right-invariant vector fields, which leads to $[a,b]=-\ad_a b$ for $a,b \in \g$, where $\ad_a b = \frac{\partial}{\partial t \partial \epsilon}|_{t,\epsilon =0} \exp(ta)\exp(\epsilon b) \exp(ta)^{-1}$ denotes the standard adjoint action (as opposed to \cite{DK} which uses left-invariant vector fields).
\end{example}

\subsection{The construction of the spray groupoid}\label{subsection:spray_groupoid}

For the rest of this section let $(A,[\cdot,\cdot],\rho)$, $q:A\to M$ be a Lie algebroid, and $V$ be a Lie algebroid spray on $A$. Using $V$ we shall construct a local Lie groupoid $G_V\rightrightarrows M$ integrating $A$. The total space is $G_V=A$, the space of units is $M$ identified with the zero section of $A$, and the source map is the bundle projection $\sigma=q:A\to M$. The target map and the inversion map are defined on the open neighborhood of $M$ on which the flow of $V$ is defined up to $t=1$:
\begin{align*}
&\tau:G_V\dto M,  \ \ \ \tau(a):=q(\phi^1_V(a)), \\
& \iota:G_V\dto G_V,  \ \ \ \iota(a):=-\phi^1_V(a)=\phi_V^{-1}(-a) \ \ (\textrm{see}\ (\ref{eq:flowV})).
\end{align*}
The groupoid axioms involving only these maps are easily verified: $\sigma$ and $\tau$ are surjective submersions, and for $x\in M$ and $a\in A$
\begin{align*}
&\tau(x)=q(\phi_V^1(0_x))=q(0_x)=x,\\
&\iota(x)=-\phi_V^1(0_x)=x,\\
&\sigma\circ\iota(a)=q(-\phi_V^1(a))=q(\phi_V^1(a))=\tau(a),\\
&\iota\circ\iota(a)=-\phi^1_V(-\phi^1_V(a))=\phi^{-1}_V\circ\phi^1_V(a)=a,
\end{align*}
where we have used that $\phi^t_V$ fixes $M$, and equation \eqref{eq:flowV}.

The main theorem of this subsection is that $G_V$ carries a natural multiplication, defined by the ODE \eqref{eq:ODE} below, that makes it into a local Lie groupoid integrating $A$, and has Maurer-Cartan form the realization form $\theta$ associated to $V$. The starting point is to observe that there is a neighborhood $U_{\mu}$ of $M$ in $A\tensor[_\sigma]{\times}{_\tau}A$, such that for all $(a,b)$ in this neighborhood the ODE
\begin{equation}
\label{eq:ODE}
\theta\Big(\frac{d}{dt}k_t\Big)=\phi_V^t(a),  \ \ k_0=b,
\end{equation}
has its unique solution $t\mapsto k_t=k(t,a,b)\in A_{\sigma(b)}$ defined for all $t\in[0,1]$. Indeed, 
we know that there is a neighbourhood $U$ of $M\subset A$ on which  bundle map $\theta: T^qA|_U \to A$ is a fiberwise linear isomorphism. Hence, for $a,b$ close enough to $M$, the ODE has a unique solution defined for $t$ in a neighborhood of $0$. Now, when $a=x\in M$, we have that $k_t=b$ is a solution for all $t\in \R$. For $b=x\in M$, $k_t=ta$ is a solution for all $t$ such that $\theta_{ta}$ is defined; this follows from equation (\ref{eq:theta_Euler2}). This implies the existence of such an open neighborhood $U_{\mu}$.

\begin{theorem}\label{theorem:thm2}
The maps $\sigma=q$, $\tau=q\circ \phi_V^1$, $\iota=-\phi_V^1$, together with the multiplication map 
\[\mu: G_{V}\tensor[_\sigma]{\times}{_\tau}G_V\dto G_V, \ \ \ \mu(a,b):=k(1,a,b),\]
where $k(t,a,b)$ is the solution of the ODE \eqref{eq:ODE}, 
define the structure maps of a local Lie groupoid on $(G_V,M,\sigma)=(A,M,q)$. The Maurer-Cartan form of this local groupoid coincides with the realization form $\theta$ associated to $V$; in particular, $G_V$ is a local Lie groupoid integrating $A$.
\end{theorem}

\begin{definition}\label{def:spray_groupoid} The local Lie groupoid $G_V$ is called the \textbf{spray groupoid} of $A$ associated to $V$.
\end{definition}

First, we detail immediate properties of $\mu$ and then we prove Theorem \ref{theorem:thm2}. As we saw before, $k(t,a,\sigma(a))=ta$ and $k(t,\tau(b),b)=b$, thus, the unit axioms hold:
\[\mu(a,\sigma(a))=a,\ \ \ \ \mu(\tau(b),b)=b.\]
Since $k_t\in A_{\sigma(b)}$ and $\tau(k_t)=q(\phi_V^t(a))$, at $t=1$, we obtain the axioms:
\[\sigma(\mu(a,b))=\sigma(b), \ \ \ \tau(\mu(a,b))=\tau(a).\]

Next, we describe the path $k(t,a,b)$ in terms of $\mu$:
\begin{lemma}\label{lema: rescaling-multiplication}
For $(a,b)\in U_{\mu}\subset G_{V}\tensor[_\sigma]{\times}{_\tau}G_V$ and $t\in [0,1]$, we have that $k(t,a,b)=\mu(ta,b)$.
\end{lemma}
\begin{proof}
Fix $s\in [0,1]$. Using (\ref{eq:flowV}), we calculate
\[\theta\Big(\frac{d}{dt}k(st,a,b)\Big)=s \theta\Big(\frac{d}{du}k(u,a,b)\big|_{u=st}\Big)=s\phi_V^{st}(a)=\phi^t_V(sa).
\]
Thus, the curves $t\mapsto k(st,a,b)$ and $t\mapsto k(t,sa,b)$ satisfy the same ODE and start at $b$. So they are equal and, for $t=1$, we obtain the conclusion $\mu(sa,b)=k(1,sa,b)=k(s,a,b)$.
\end{proof}

In particular, we obtain that
\begin{equation}\label{eq:theta}
\theta\Big(\frac{d}{dt}\mu(ta,b)\Big)=\theta\Big(\frac{d}{dt}ta\Big)=\phi_V^t(a).
\end{equation}
The following result shows that $\theta$ plays the role of the Maurer-Cartan form of $\mu$:

\begin{lemma}\label{lemma:right inv}
 $\theta$ is right invariant for $\mu$:
\[\theta\Big(\frac{d}{d\epsilon}\mu(a+\epsilon v,b)\big|_{\epsilon=0}\Big)=
\theta\Big(\frac{d}{d\epsilon}(a+\epsilon v)\big|_{\epsilon=0}\Big)=\theta_a(v),\]
for all $(a,b)\in U_{\mu}$ and all $v\in A_{\sigma(a)}$.
\end{lemma}
\begin{proof}
Let $J$ be a small interval around $0$ such that for all $(t,\epsilon)\in [0,1]\times J$ we have that $\mu(t(a+\epsilon v),b)\in A_{\sigma(b)}$ is defined. Then the map
\[d(\mu(t(a+\epsilon v),b))=\frac{d}{dt}\mu(t(a+\epsilon v),b)dt+\frac{d}{d\epsilon}\mu(t(a+\epsilon v),b)d\epsilon: T([0,1]\times J)\to T^{\sigma}G_V\]
is a Lie algebroid morphism. Composing this with the Lie algebroid map $\theta:T^{\sigma}G_V\dto A$, and applying (\ref{eq:theta}), we obtain that the following is a Lie algebroid morphism:
\[\phi^t_V(a+\epsilon v)dt+\theta\Big(\frac{d}{d\epsilon}\mu(t(a+\epsilon v),b)\Big)d\epsilon: T([0,1]\times J)\to A.\]
The coefficient of $d\epsilon$ vanishes for $t=0$, so, by the uniqueness property in Lemma \ref{Lemma_CF} and by Lemma \ref{lemma:old_theta}, we obtain
\[\theta\Big(\frac{d}{d\epsilon}\mu(t(a+\epsilon v),b)\Big)=t\theta_{t(a+\epsilon v)}(v),\]
which, for $\epsilon=0$ and $t=1$, gives the conclusion.
\end{proof}

Although in Lemma \ref{lemma:right inv} we used paths of the particular form $\epsilon\mapsto (a+\epsilon v)$, the result is about the differential of right multiplication $a\mapsto \mu(a,b)$, which is defined on an open set around $\tau(b)$ in $A_{\tau(b)}$. Therefore, we can reformulate the lemma as follows: for any smooth curve $\epsilon\mapsto a_\epsilon$, such that $(a_{\epsilon},b)\in U_{\mu}$, we have that
\[\theta\Big(\frac{d}{d\epsilon}\mu(a_{\epsilon},b)\Big)=\theta\Big(\frac{d}{d\epsilon}a_{\epsilon}\Big).\]
This implies that:
\begin{lemma}
The map $\mu$ is associative on a neighborhood of $M$ in $G_V^{(3)}$.
\end{lemma}
\begin{proof}
Consider $(a,b,c)\in G_V^{(3)}$, close enough to $M$ such that $\mu(\mu(ta,b),c)$ is defined for all $t\in [0,1]$. By right invariance of $\theta$, and by \eqref{eq:theta}, we have that
\[\theta\Big(\frac{d}{dt}\mu(\mu(ta,b),c)\Big)=\theta\Big(\frac{d}{dt}\mu(ta,b)\Big)=\phi^t_V(a).\]
Hence the two curves $t\mapsto \mu(\mu(ta,b),c)$ and $t\mapsto \mu(ta,\mu(b,c))$ satisfy the same ODE and start at $\mu(b,c)$; thus they are equal. For $t=1$, we obtain associativity: $\mu(\mu(a,b),c)=\mu(a,\mu(b,c))$.
\end{proof}

Finally, we check the law of inverses. To that end, we first prove the following
\begin{lemma}\label{lma:magic}
Given $t,s \in \R$, the identity \[  \mu(t\phi_V^s(a),sa)=(t+s)a \]
holds for all $a\in A$ close enough to $M\subset A$.
\end{lemma}
\begin{proof}
The idea is to show that both curves $t \mapsto \mu(t\phi_V^s(a),sa)$ and $t \mapsto (t+s)a$ satisfy the same ODE and have the same initial condition at $t=0$ so that they must agree. (Notice that for the first curve to be defined, we need $a$ close to $M\subset A$.) Taking their velocities and composing with $\theta$ we get:
\[\theta(\frac{d}{dt}(t+s)a)=\theta(\frac{d}{du}ua|_{u=t+s})=\phi_V^{t+s}(a)\] 
and, by \eqref{eq:theta},
\[\theta\left(\frac{d}{dt}\mu(t\phi_V^s(a),sa)\right)=\phi_V^t(\phi_V^s(a))=\phi^{t+s}_V(a).\]
Taking $a$ close enough to $M\subset A$ so that the curves lie in the domain in which $\theta$ is a fiberwise isomorphism, we conclude that both curves satisfy the same ODE \eqref{eq:ODE}. Since both curves also start at the same point $sa\in A$ for $t=0$, the lemma follows.
\end{proof}

Coming back to the law of inverses, 
since $\iota|_M=\mathrm{id}_M$, then $(a,\iota(a)),(\iota(a),a)\in U_{\mu}$ for all $a$ close enough to $M$. Moreover, we can take such an $a$ close enough to $M\subset A$ so that Lemma \ref{lma:magic} with $t=-1$ and $s=1$ also applies yielding
\[ \mu(\iota(a),a) = \mu(-\phi_V^1(a),a) = 0 \cdot a = \sigma(a).\]
Analogously, considering $s=-1$, $t=1$ in the lemma above and denoting $a=\phi^1_V(b)$,
\[ \mu( b, \iota(b)) = \mu(\phi_V^{-1}\phi_V^1(b),-\phi_V^1(b)) = \mu(\phi_V^{-1}(a),-a) = 0 \cdot a = \tau(b).\]
The law of inverses is thus proven.

We conclude that $\mu$ endows $G_V$ with the structure of a local groupoid. By Lemma \ref{lemma:right inv}, and by the fact that $\theta_x=\mathrm{id}_{A_x}$ for all $x\in M$, we have that $\theta$ is the Maurer-Cartan form of $G_V$. Thus, we have proven Theorem \ref{theorem:thm2}.

\begin{remark}\label{rmk:tubstronspray}
 Notice that the spray groupoid $G_V\rightrightarrows M$ just defined carries the natural tubular structure $\varphi$ (in the sense of Section \ref{subsec:tubular}) given by the identity map of $G_V = A$.
\end{remark}

\begin{remark}
Let us consider the spray groupoid $G=G_V$ and, in the context of Remark \ref{rmk:factorizatgd}, the following diagram of local Lie groupoid maps
 \begin{equation}\label{eq:division_diagram_spray}
\xymatrixrowsep{0.4cm}
\xymatrixcolsep{1.2cm}
\xymatrix{
 G\tensor[_\sigma]{\times}{_\sigma}G \ar@{^{(}->}[r]\ar@<-.5ex>[d] \ar@<.5ex>[d] & \sigma^!G\ar@{-->}[r]^{\Ad_{\bar\beta}}\ar@<-.5ex>[d] \ar@<.5ex>[d] & \sigma^!G\ar@<-.5ex>[d] \ar@<.5ex>[d]\ar[r]^{\mathrm{pr}_2}& G\ar@<-.5ex>[d] \ar@<.5ex>[d] \\
 G\ar[r]^{\mathrm{id}}&G\ar@{-->}[r]^{-\iota}&G \ar[r]^\sigma & M}
 \end{equation}
where $\bar \beta(a)=(-\iota(a),a,a)=(\phi_V^1(a),a,a)$ defines a bisection of $\sigma^!G$. The above diagram is a slight modification of \eqref{eq:division_diagram} in which $\beta$ is replaced by $\bar \beta$ but the overall composition still yields the division map \eqref{division_map}. Moreover, \eqref{eq:division_diagram_spray} integrates the diagram of Lie algebroid morphisms \eqref{defi:MC_diagram} thus providing a conceptual explanation for the construction of $\theta$. To see this, the only non-trivial verification needed is that $\Ad_{\bar \beta}$ integrates $\phi^1_{\widehat V}$, having in mind that $q^!A$ is the Lie algebroid of $\sigma^!G$ via the identification
\begin{equation}\label{eq:identpull} T_bA\times_{TM}A \ni (u,a) \mapsto (u,a,0_b) \in T_{(b,\sigma(a),b)}(G _\sigma\times_\tau G _\sigma\times_{\sigma} G).\end{equation} By the general formula \eqref{eq:bisections} for differentiating bisections of flows, it is in turn enough to check that $\bar \beta( a)$ coincides with the flow up to time $t=1$ of the right-invariant vector field $\widehat V ^R \in \X(\sigma^!G)$, which is induced by the section $\widehat V$ of $q^!A$, when applied to the identity element $(a, \sigma(a) ,a)\in \sigma^!G$ defined by $a\in G$. Writing $c(t) = (\phi^t_V(a),ta,a)$, Lemma \ref{lma:magic} implies that \[ c(t+s) = (\phi_V^{t+s}(a), s\phi^t_V(a), \phi^t_V(a)) \cdot c(t),\]
where the $\cdot$ denotes multiplication in $\sigma^!G$ (c.f.\ Remark \ref{rmk:factorizatgd}). It then follows that $c(t)$ is the integral curve of $\widehat V ^R$ starting at $c(0)=(a,\sigma(a),a)$:
\[ \frac{d}{ds}|_{s=0} c(t+s) =  \frac{d}{ds}|_{s=0}  (\phi_V^{s}(\phi^t_V(a)), s\phi^t_V(a), \phi^t_V(a) ) \cdot c(t) = \widehat V ^R|_{c(t)},\]
since $\widehat V ^R|_{(b,\sigma(b),b)}= \frac{d}{ds}|_{s=0}(\phi_V^s(b),sb,b) \in T_{(b,\sigma(b),b)} \sigma^!G$ by \eqref{eq:identpull}. Thus $c(1) = (\phi^1_V(a),a,a) = \bar \beta(a)$ as desired.
\end{remark}

\begin{example}\label{ex:Lie_groups}
 Let $A=\g$ be a Lie algebra, and let us choose the zero spray $V=0$ as in Examples \ref{ex:alg}, \ref{ex:lie2}. Using the specific formula for the realization form $\theta$ computed in \eqref{eq:thetaonLie}, the ODE \eqref{eq:ODE} defining $\mu$ on the spray group $G_V=\g$ becomes
  \[\frac{d}{dt}k_t=\frac{-\ad_{k_t}}{e^{-\ad_{k_t}}-I}(a),\ \ k_0=b.\]
  It is a standard computation to deduce the Baker-Campbell-Hausdorff series (or Dynkin's formula) from the above differential equation 
$$\mu(a,b)=k_1= a + b -\frac{1}{2} [a,b] + ...
  $$ 
See \cite[Section 1.7]{DK} where the bracket appears replaced by its opposite due to the difference in conventions recalled in Example \ref{ex:lie2}.
\end{example}

\subsection{Local integration of algebroid morphisms to the spray groupoid}

Remark \ref{rmk:tubstronspray}, Theorem \ref{thm:integr_morph} and equation \eqref{eq:theta} imply the following:

\begin{corollary}\label{corollary:morphisms}
Let $f:A_1\to A_2$ be a Lie algebroid morphism covering $f_M:M_1\to M_2$. Let $G_V\rightrightarrows M_1$ be the spray groupoid corresponding to a spray $V$ for $A_1$, and let $G_2\rightrightarrows M_2$ be any local Lie groupoid integrating $A_2$. The local Lie groupoid map
\[F:G_V\dto G_2\]
integrating $f$ from Theorem \ref{thm:integr_morph} is given by $F(a)=k(1,a)$,
where, for $a\in G_V$ close enough to $M_1$,
$t\mapsto k_t=k(t,a)$ is the solution of the ODE
\begin{equation}\label{eq:morphisms}
\theta_{G_2}\Big(\frac{d}{dt}k_t\Big)=f(\phi_V^t(a)), \ \ \ k_t\in\sigma_2^{-1}(f_{M}(\sigma_1(a))),\ \ \  k_0=f_{M}(\sigma_1(a))
\end{equation}
\end{corollary}

\begin{example}\label{ex:trivint}
 Let $f:A_1 \to A_2$ be a Lie algebroid morphism, and assume that there are sprays $V_1$ and $V_2$, on $A_1$ and $A_2$ respectively, such that they are $f$-related. In this case, it follows from equation (\ref{eq:theta_Euler2}) that $k(t,a) = t f(a)$ is a solution to equation equation (\ref{eq:morphisms}); therefore $F:=f:G_{V_1}\dto G_{V_2}$ is the local groupoid map integrating $f$. 

However, let us remark that for a general Lie algebroid map, related sprays might not exist; for example, one can check that this is the case for $f=dg:T\R\to T\R$, where $g(x)=x^2$.
\end{example}

\begin{example}\emph{(Integrating $1$-cocycles)}\label{rmk:cocycles}
Let $f:A\to \R$ be a Lie algebroid morphism, where $\R$ is a trivial Lie algebra, viewed as a Lie algebroid over a point. Note that $f$ represents a $1$-cocycle for $A$. In this case, equation \eqref{eq:morphisms} reduces to 
 $$ \frac{d}{dt} k_t = f(\phi_V^t(a)), \ k_0 = 0.$$
Hence, the local groupoid map $F:G_V \dto \R$ ($\R$ seen as a $1$-dimensional Lie group) is given by
 $$ F(a) = \int_0^1 f(\phi_V^t(a)) \ dt,$$
 which defines a local $1$-cocycle on $G_V$ integrating $f$. Compare with the inverse \eqref{eq:our_formula} of the van Est map at the level of 1-cochains: for the spray groupoid $\theta(\frac{d}{dt}ta)=\phi^t_V(a),$ hence the two formulas coincide. Similar formulas hold for multiplicative forms and other multiplicative geometric structures near the identities; these will be detailed elsewhere \cite{CMSbis}.
\end{example}

For $f=\mathrm{id}$, the induced map $F$ is called the {\it exponential map} corresponding to the spray:

\begin{definition}\label{defi:exp}
Let $A$ be a Lie algebroid with a spray $V$. Assume that $G$ is a local Lie groupoid integrating $A$ with Maurer-Cartan form $\theta_G$. For $a\in A_x$, close enough to $M$, the ODE
\[\theta_G\Big(\frac{d}{dt}k_t\Big)=\phi^t_V(a), \ \ k_0=x,\]
has a solution $t\mapsto k_t=k(t,a)\in \sigma^{-1}(x)$, for all $t\in [0,1]$. The \textbf{spray exponential} is the map
\[\exp_V:A\dto G, \ \ \exp_V(a):=k(1,a).\]
\end{definition}

Note that $\exp_V$ defines a diffeomorphism near $M$ (its differential along $M$ is the identity). The results of this section imply the following characterization of the spray groupoid structure:

\begin{corollary}\label{coro:exp}
Let $A$ be a Lie algebroid with a spray $V$. The germ of the spray groupoid structure on $(A,M,q)$ is uniquely characterized by the property that, for any local groupoid $G$ integrating $A$, the associated spray exponential $\exp_V:G_V\dto G$ defines a local Lie groupoid map (in this case, the germ of $\exp_V$ automatically defines an isomorphism between the germs of $G_V$ and $G$).
\end{corollary}

\begin{remark}
 In the case that $A$ is integrable, i.e.\ if the Weinstein groupoid $G(A)=P(A)/\mathcal{F}$ is smooth 
(see \cite{CF1}), then the exponential map associated to $G(A)$ and $V$ factors as $\exp_V=\pi\circ\hexp$, where $\hexp:A\dto P(A)$ was described in equation \eqref{eq:hatexp} of the  introduction, and $\pi:P(A)\to G(A)$ is the quotient map.
\end{remark}

\subsection{Tubular structures versus spray exponentials}\label{subsection:tubular groupoids}

In this final subsection, we show that a tubular structure $\varphi: A \dto G$ naturally induces a Lie algebroid spray $V_\varphi$ on $A$. The resulting local Lie groupoid $G_{V_{\varphi}}$ can be thought of as providing a ``spray groupoid approximation'' to $G$. We provide necessary and sufficient conditions for $\varphi:G_{V_{\varphi}} \dto G$ to be a local Lie groupoid map.

Let $G\rightrightarrows M$ be a local Lie groupoid with Lie algebroid $A$, and $\varphi: A \dto G$ a tubular structure as defined in Section \ref{subsec:tubular}. Following the case of a spray groupoid, in which the identity $A \to G_V$ defines a tubular structure, we consider 
\begin{equation}\label{eq:structure equation}
\lambda_t:A\dto A, \ \ \lambda_t(a):=\theta_G\Big(\frac{d}{dt}\varphi(ta)\Big).
\end{equation}
The map $(t,a)\mapsto \lambda_t(a)$ is defined on an open subset of $\R\times A$ containing both $\R\times M$ and $\{0\}\times A$. Since $\varphi$ is a tubular structure on $G$, we have that $\lambda_0=\mathrm{id}$. 
\begin{lemma}\label{lma:sprayfromtubular}
The derivative of $\lambda_t$ at $t=0$
\[ V_\varphi \in \mathfrak{X}(A), \ \ V_\varphi(a):=\frac{d}{dt}\lambda_t(a)\big|_{t=0}\]
defines a Lie algebroid spray for $A$.
\end{lemma}
\begin{proof}
First, note that $\lambda_t(a)$ is an $A$-path: this is equivalent to $\lambda_t(a) dt:T[0,1]\to A$ being a Lie algebroid map, which follows since it is the composition of the Lie algebroid maps \[\theta_G:T^\s G \dto A\ \ \textrm{ and}\ \ \frac{d}{dt}( \varphi(t a)) \  dt:T[0,1]\to T^\s G.\] 
Using $\s \circ \varphi = q$, the $A$-path condition at $t=0$ gives the second spray condition: $dq (V_\varphi(a))=\rho(a)$.\\
Next, note that $\lambda_t(a)$ satisfies the analog of equation (\ref{eq:flowV}):
\[\lambda_t(sa)=\theta_G\Big(\frac{d}{dt}\varphi(sta)\Big)=s\theta_G\Big(\frac{d}{du}\varphi(ua)\big|_{u=st}\Big)=s\lambda_{st}(a).\]
This implies that
$m_{s}^{-1}\circ \lambda_{t}\circ m_s=\lambda_{st}.$
Taking the derivative at $t=0$, we obtain the first spray condition $m_s^*(V_\varphi)=sV_\varphi$.
\end{proof}

Let $\exp_{V_\varphi} : A \dto G$ be the spray exponential associated to $V_\varphi$ (cf.\ Definition \ref{defi:exp}). The following proposition states criteria for $\varphi:G_{V_{\varphi}}\dto G$ to be a local groupoid map.

\begin{proposition}\label{prop:uniqspraystruct}
With the above notation, the following are equivalent:
\begin{enumerate}
\item $\varphi:G_{V_{\varphi}}\dto G$ is a local groupoid map;
\item $\exp_{V_\varphi}=\varphi$ near the zero section;
\item the isotopy $\lambda_t$ coincides with the flow of $V_\varphi$ near the zero section.  
\end{enumerate}
\end{proposition}

\begin{proof}
By Corollary \ref{coro:exp}, $2.$ implies $1.$

Assume that $3.$ holds: $\lambda_t =\phi_{V_\varphi}^t $. By Lemma \ref{lma:sprayfromtubular}, $V_\varphi$ is a spray for $A$. On the other hand, for small enough $a\in A$, we have $\theta_G(\frac{d}{dt}\varphi(ta))=\lambda_t(a)=\phi_{V_\varphi}^t(a)$. Therefore $k_t=\varphi(ta)$ satisfies the ODE from Definition \ref{defi:exp} of the spray exponential map. Taking $t=1$, we conclude $\exp_{V_\varphi}(a) = \varphi(a)$ for any $a\in A$ close enough to the zero section, thus $2.$ holds. 

Finally, assume that $1.$ holds. Then $\varphi$ intertwines the Maurer-Cartan forms of the two groupoids, and so: 
\[ \lambda_t(a) = \theta_G\left(\frac{d}{dt}\varphi(ta)\right) = \theta\left(\frac{d}{dt}(ta)\right)= \phi_{V_{\varphi}}^t(a),\]
where we have used that $\varphi$ induces the identity at the level of Lie algebroids. Thus $3.$ holds. 
\end{proof}
%
%
%
%
%
%
%
%

\begin{appendix}
\section{Differentiating integrated cochains}\label{appendix}
In this appendix, we consider the differentiation of the local Lie groupoid cochains obtained by integration through the map $\Psi$ defined in Section \ref{sec:cochains}. The main result is Lemma \ref{lma:DvE} below which is used in the proof of item \ref{e2} of Proposition \ref{prop:ve}.

Following the notation of Section \ref{sec:cochains}, let $\alpha \in \Gamma( \wedge^p A^*)$ be an algebroid $p$-cochain and let $\Psi(\alpha): G^{(p)}\dto \R$ be the corresponding local groupoid cochain. In order to apply the van Est map, we consider sections $a_1 \comas a_k \in \Gamma(A)$ and compute the iterated differentiation $ D_{a_p}\cdots D_{a_1} \Psi(\alpha)$. We do this inductively on the number of derivatives, starting by computing $D_{a_1}\Psi(\alpha)$.
 
Let us fix composable arrows $g_2 \comas g_p \in G^{(p-1)}$ in a sufficiently small neighborhood of $M$, and consider $h_1(\e) = \e a_1(\t(g_2))$ so that
\begin{align}
\nonumber D_{a_1}\Psi(\alpha) (g_2 \comas g_p) &= \frac{d}{d\e}|_{\e=0}\Psi(\alpha)(h_1(\e),g_2 \comas g_p) \\ \nonumber &= \frac{d}{d\e}|_{\e=0}
\int_{I^p} (\theta^*\alpha)\left( \partial_{1}\gamma_{h_1(\e), g_2 \comas g_p} \comas {\partial}_{p}\gamma_{h_1(\e), g_2 \comas g_p} \right)  dt_1\dots dt_p.
\end{align}
Using the relations
\begin{equation} \label{eq:homgamma}
\gamma_{\e g_1, g_2\comas g_p}(t_1\comas t_p) = \gamma_{g_1\comas g_p}(\e t_1, t_2\comas t_p),\ \ \gamma_{\tau(g_2), g_2\comas g_p}(t_1\comas t_p) = \gamma_{g_2\comas g_p}(t_2\comas t_p),
\end{equation}
where $\tau(g_2)=\sigma(g_1)=0\cdot g_1$, it follows that, for $a\in \Gamma(A)$ and small $(g_2 \comas g_n) \in G^{(n)}$,
\begin{align}
\nonumber \partial_{1}\gamma_{\e a(\t g_2) ,g_2 \comas g_n}(t_1 \comas t_{n}) = \e \cdot \partial_{1}\gamma_{a(\t(g_2)),g_2 \comas g_n}(\e t_1 , t_2\comas t_{n}), \\ 
\partial_{j}\gamma_{\t(g_2),g_2 \comas g_n}(t_1,t_2 \comas t_n) = 
\partial_{j-1}\gamma_{g_2 \comas g_n}(t_2 \comas t_n) ,\ j>1. \label{eq:thetaids}
\end{align}

Applying \eqref{eq:thetaids}, we can compute $\frac{d}{d\e}|_{\e=0}$ (equivalently, apply the substitution $u=\e t_1$ as in Example \ref{ex:1cochains}) yielding
\begin{align}
\nonumber &D_{a_1} \Psi(\alpha) (g_2 \comas g_p) = \\ \nonumber &= 
\int_{I^{p-1}} (\theta^*\alpha)\left( \partial_{1}\gamma_{a_1(\t (g_2)),g_2 \comas g_p}(0,t_2\comas t_p), \partial_{1}\gamma_{g_2 \comas g_p}(t_2\comas t_p)  \comas {\partial}_{p-1}\gamma_{g_2 \comas g_p}(t_2\comas t_p) \right)  dt_2\dots dt_p,
\end{align}
where $\int_I dt_1 =1$ has been factored out.
Above, the first factor inside $\alpha$ is of a different nature from the rest: it is the only one depending on $a_1$ while the others only depend on the fixed string $g_2 \comas g_p$. We introduce the following general notation for such terms: given $a \in \Gamma(A)$, small $(k_1 \comas k_n) \in G^{(n)}$ and any $(s_1 \comas s_n) \in I^n$, 
\begin{equation} \label{eq:defv}
 v^{a}_{k_1 \comas k_n}(s_1\comas s_n) := \partial_{1}\gamma_{a(\t (k_1)),k_1 \comas k_n}(0,s_1\comas s_n)\in T^{\sigma}G_{\gamma_{k_1\comas k_n}(s_1 \comas s_n)}.
\end{equation}
Before stating the general inductive result, let us compute one more derivative fixing $g_3 \comas g_p \in G^{(p-2)}$ and letting $h_2(\e) = \e a_2(\t(g_3))$,
\begin{align}
\nonumber D_{a_2}&D_{a_1} \Psi(\alpha) (g_3 \comas g_p) = \frac{d}{d\e}|_{\e=0} D_{a_1} \Psi(\alpha)(h_2(\e),g_3 \comas g_p) \\
 \nonumber =& 
 \frac{d}{d\e}|_{\e=0}\int_{I^{p-1}}(\theta^* \alpha)\left(  v^{a_1}_{h_2(\e),g_3 \comas g_p}, \partial_{1}\gamma_{h_2(\e),g_3 \comas g_p}  \comas {\partial}_{p-1}\gamma_{h_2(\e),g_3 \comas g_p} \right)(t_1\comas t_{p-1})  dt_1\dots dt_{p-1},
\end{align}
where we have relabeled the integration variables. Using \eqref{eq:thetaids} we compute $\frac{d}{d\e}|_{\e=0}$, yielding the integral over $dt_1 \dots dt_{p-1}$ of
\[
(\theta^*\alpha)\left( v^{a_1}_{h_2(0),g_3 \comas g_p}(t_1\comas t_{p-1}), v^{a_2}_{g_3 \comas g_p}(t_2 \comas t_{p-1}),
\partial_{1}\gamma_{g_3 \comas g_p}(t_2\comas t_{p-1})  \comas {\partial}_{p-2}\gamma_{g_3 \comas g_p}(t_2 \comas t_{p-1}) \right)  
\]
The key step for recognizing an inductive structure in our computation is to note that (recall that $M$ is identified with the unit section) 
\begin{equation} \label{eq:0in2}
 \gamma_{g_1,\t(g_3),g_3\comas g_n}(t_1,t_2\comas t_n) = \gamma_{g_1,g_3\comas g_n}(t_1\cdot t_2, t_3\comas t_n), 
\end{equation}
for any small composable string $(g_1,g_3,g_4 \comas g_n) \in G^{(n-1)}$ and any $\ (t_1,t_2 \comas t_n) \in I^n$.
Using \eqref{eq:0in2} and the notation \eqref{eq:defv}, we get that
\begin{align} \label{eq:vaprop}
 v^{a}_{\t ( g_3 ),g_3 \comas g_n}(t_1\comas t_{n-1}) = t_1 \cdot v^{a}_{g_3 \comas g_n}(t_2\comas t_{n-1})
\end{align}
for any $a\in \Gamma(A)$, small $(g_3 \comas g_n) \in G^{(n-2)}$ and $(t_1\comas t_{n-1}) \in I^{n-1}$.
Thus,  since $h_2(0) = \t(g_3)$,
\begin{align}
\nonumber D_{a_2}D_{a_1} & \Psi(\alpha) (g_3 \comas g_p) =\\
\nonumber &= \int_{I^{p-1}} t_1 \cdot (\theta^* \alpha)\left(  v^{a_1}_{g_3 \comas g_p}, v^{a_2}_{g_3 \comas g_p} , {\partial}_{1}\gamma_{g_3 \comas g_p} \comas {\partial}_{p-2}\gamma_{g_3 \comas g_p} \right) (t_2 \comas t_{p-1})  dt_1\dots dt_{p-1}
\end{align}
from which $\int_I t_1 \ dt_1 = 1/2$ factors out yielding
\begin{align}
\nonumber D_{a_2}D_{a_1} \Psi(\alpha) (g_3 \comas g_p) =\frac{1}{2}
 \int_{I^{p-2}}(\theta^* \alpha)\left( v^{a_1}_{g_3 \comas g_p}, v^{a_2}_{g_3 \comas g_p} , {\partial}_{1}\gamma_{g_3 \comas g_p} \comas {\partial}_{p-2}\gamma_{g_3 \comas g_p} \right).
\end{align}

Continuing by induction we obtain the following:
\begin{lemma}\label{lma:DvE}
For $a_1 \comas  a_k\in \Gamma(A)$ and $(g_{k+1} \comas g_p)\in G^{(p-k)}$ small enough, we have:
\begin{align}\label{eq:Ds}
\nonumber D_{a_k}&\cdots D_{a_1} \Psi(\alpha) (g_{k+1} \comas g_p)=\\ 
 &= \frac{1}{k!} \int_{I^{p-k}} (\theta^*\alpha)\left(v^{a_1}_{g_{k+1}\comas g_p} \comas v^{a_k}_{g_{k+1}\comas g_p}, \partial_{1} \gamma_{g_{k+1}\comas g_p} \comas  \partial_{p-k} \gamma_{g_{k+1}\comas g_p}\right)
\end{align}
where $v^{a}_{g_{k+1} \comas g_p}:I^{p-k} \to T^{\sigma}G$ is defined in \eqref{eq:defv}.
\end{lemma}

\begin{proof}
We will consider an induction over $k\geq 1$ recalling that the case $k=1$ was already worked out above.
Assume now \eqref{eq:Ds} and compute
\[ D_{a_{k+1}}\cdots D_{a_1} \Psi(\alpha) (g_{k+2} \comas g_p) = \frac{d}{d\e}|_{\e=0} D_{a_k}\cdots D_{a_1} \Psi(\alpha) (\e a_{k+1}(\t(g_{k+2})),g_{k+2}\comas  g_p).\]
To use \eqref{eq:Ds} on the right hand side above, let us introduce the following notation: 
\[ w^{a_1 \comas a_k}_{g_{k+1} \comas g_{p}}: = v^{a_1}_{g_{k+1}\comas  g_p}\wedge \dots \wedge v^{a_k}_{g_{k+1}\comas g_p}\]
and $b:= a_{k+1}(\t(g_{k+2}))$, so that
\begin{equation}\label{YAE}
 k!  D_{a_{k+1}}\cdots D_{a_1} \Psi(\alpha) (g_{k+2}\comas g_p)= \frac{d}{d\e}|_{\e=0}\int_{I^{p-k}} (\theta^*\alpha)(w^{a_1\comas a_k}_{\e b, g_{k+2}\comas g_{p}} \wedge  d \gamma_{\e b, g_{k+2}\comas g_p}).
\end{equation}
Using \eqref{eq:homgamma} it follows that
\[ d\gamma_{\e b,g_{k+2}\comas g_p}(t_1\comas t_{p-k}) = \e \cdot  d\gamma _{b,g_{k+2}\comas g_p}(\e t_1\comas t_{p-k}),\]
and that $\gamma_{b,g_{k+2}\comas g_p}(0,t_2\comas t_{p-k})=\gamma_{g_{k+2}\comas g_p}(t_2\comas t_{p-k})$. Using these to compute \eqref{YAE}, gives
\begin{align*}
  k!  D_{a_{k+1}}&\cdots D_{a_1} \Psi(\alpha) (g_{k+2}\comas g_p)=\\
&=\int_{I^{p-k}} (\theta^*\alpha)\big(w^{a_1\comas a_k}_{\tau(g_{k+2}),g_{k+2}\comas g_{p}}(t_1\comas t_{p-k})\wedge d\gamma_{g_{k+2}\comas g_p}(t_2\comas t_{p-k})\big) \ dt_1\dots dt_{p-k}.
\end{align*}
Applying \eqref{eq:vaprop} we get 
\[w^{a_1\comas a_k}_{\tau(g_{k+2}),g_{k+2}\comas g_{p}}(t_1\comas t_{p-k}) = t_1^k \cdot w^{a_1\comas a_k}_{g_{k+2}\comas g_{p}}(t_2, t_3\comas t_{p-k}).\]
The lemma thus follows 
by noticing that $\int_0^1 t_1^k\ dt_1= 1/(k+1)$ factors out.
\end{proof}
\end{appendix}

\end{document}